\documentclass[11pt, english, a4paper]{amsart}

\usepackage{graphicx}
\usepackage{amsmath}
\usepackage{amssymb}
\usepackage[latin1]{inputenc}
\usepackage{bm} 
\usepackage[T1]{fontenc}
\usepackage{amsfonts}
\usepackage{mathrsfs}
\usepackage{yfonts}
\usepackage{url}
\usepackage[osf]{mathpazo}  
\usepackage[toc,page]{appendix}

\usepackage{mathtools}

\usepackage{graphicx}
\usepackage{verbatim}
\usepackage{enumerate}	
\usepackage[english]{babel}
\usepackage{ebproof}		
\usepackage{quoting}
\usepackage{array,booktabs}
\usepackage[all]{xy}\SelectTips{cm}{}
\DeclareMathOperator{\Ker}{Ker}

 \usepackage{pifont}

\usepackage[shortlabels]{enumitem}
\usepackage{units}
\usepackage{xspace}\xspaceaddexceptions{\,}

\usepackage{booktabs,caption}
   \usepackage{tikz-cd}
\quotingsetup{font=small}	

\usepackage{todonotes}

\theoremstyle{definition}
\newtheorem{definition}{Definition}
\newtheorem{example}[definition]{Example}

\newtheorem{remark}[definition]{Remark}

\theoremstyle{plain}
\newtheorem{lemma}[definition]{Lemma}
\newtheorem{proposition}[definition]{Proposition}
\newtheorem{theorem}[definition]{Theorem}
\newtheorem{corollary}[definition]{Corollary}

\newcommand\PL{{\mathcal{P}}_{\textit{\l}}}

\newcommand\IBSL{\mathsf{IBSL}}

\newcommand{\alg}{\mathbf}
\newcommand{\class}{\mathsf}
\newcommand{\logic}{\textsc}

\newcommand{\B}{\logic{B}^{\mathrm{e}}}
\newcommand{\NBe}{\logic{NB}^{\mathrm{e}}}

\newcommand\A{{\mathbf A}}
\newcommand\C{{\mathbf C}}
\newcommand\D{{\mathbf D}}
\newcommand\Fm{\mathbf{Fm}}

\newcommand\WK{{\mathbf{WK}}}
\newcommand\WKt{\WK^{\mathrm{e}}}
\makeatletter
\providecommand*{\Dashv}{%
  \mathrel{%
    \mathpalette\@Dashv\vDash
  }%
}
\newcommand*{\@Dashv}[2]{%
  \reflectbox{$\m@th#1#2$}%
}
\makeatother

\newcommand\Ho{\mathrm{H_{0}}}

\newcommand{\Jzero}{J_{_0}}
\newcommand{\Juno}{J_{_1}}
\newcommand{\Jdue}{J_{_2}}

\newcommand{\Ji}{J_{_k}}

\newcommand\ant{\nicefrac12}

\newcommand\pair[1]{{\langle#1\rangle}}

\newcommand{\Var}{\mathnormal{V\mkern-.8\thinmuskip ar}}

\usepackage{hyperref}
 \usepackage{amsaddr}

\title[]{On the structure of Bochvar algebras}
\author{S. Bonzio}\address{Department of Mathematics and Computer Science, \\
University of Cagliari, Italy.\\ \href{mailto:stefano.bonzio@unica.it}{stefano.bonzio@unica.it}}
\author{M. Pra Baldi}
\address{FISPPA Department, University of Padua, Italy.\\ \href{mailto:michele.prabaldi@unipd.it}{michele.prabaldi@unipd.it}}

\begin{document}

\maketitle

\begin{abstract}
Bochvar algebras consist of the quasivariety $\mathsf{BCA}$ playing the role of equivalent algebraic semantics for Bochvar (external) logic, a logical formalism introduced by Bochvar \cite{Bochvar} in the realm of (weak) Kleene logics. In this paper, we provide an algebraic investigation of the structure of Bochvar algebras. In particular, we prove a representation theorem based on P\l onka sums and investigate the lattice of subquasivarieties, showing that Bochvar (external) logic has only one proper extension (apart from classical logic), algebraized by the subquasivariety $\mathsf{NBCA}$ of $\mathsf{BCA}$. Furthermore, we address the problem of (passive) structural completeness ((P)SC) for each of them, showing that $\class{NBCA}$ is SC, while $\class{BCA}$ is not even PSC. Finally, we prove that both $\mathsf{BCA}$ and $\mathsf{NBCA}$ enjoy the Amalgamation Property (AP).
\end{abstract}



The recent years have seen a renaissance of interests and studies around weak Kleene logics, logical formalisms that were considered, in the past, not particularly attractive in the panorama of three-valued logics, due to reputed ``odd'' behavior of the third-value. The late (re)discovery of weak Kleene logic regards, almost exclusively, \emph{internal} rather than \emph{external} logics: the latter, in essence, consisting of linguistic expansions of the former. More precisely, here, for external Kleene logics we understand the external version of Bochvar logic (introduced by Bochvar himself \cite{Bochvar}) and of Paraconsistent weak Kleene logic (introduced by Segerberg \cite{Segerberg65}).

The idea of considering the external connectives, thus enriching the (internal) logical vocabulary, is originally due to Russian logician D. Bochvar \cite{Bochvar}. His aim was, from the one side, to adopt a non-classical base to get rid of set-theoretic and semantic paradoxes (by interpreting them to $\ant$) and, from the other, to preserve the expressiveness of classical logic. Although his attempt failed in reaching the former purpose (as it can be shown that paradoxes resurface \cite{Urquhart2001}), the work of Bochvar has left us with a logic extremely rich in expressivity and whose potential has yet to be discovered and applied in its full capacity. Indeed, external weak Kleene logics have the advantage of limiting the infectious behavior of the third value -- the feature making them apparently little attractive -- which is confined to internal formulas only, and to recover all the consequences of classical logic in the purely external part of the language (this holding true for Bochvar external logic only). We believe that these features may turn out to be very useful in computer science and AI, providing new tools for modeling errors, concurrence and debugging.

From a mathematical viewpoint, internal weak Kleene logics show  quite  a weak connection with respect to their algebraic counterparts, as they are examples of the so-called non-protoalgebraic logics. On the other hand, a recent work \cite{Ignoranzasevera} has shown that Bochvar external logic is algebrizable with the quasivariety of Bochvar algebras (introduced in \cite{FinnGrigolia}) as its equivalent algebraic semantics. This observation gives a justified starting point for a deeper, intertwined, investigation of Bochvar external logic and Bochvar algebras, which is the main scope of the present work. 

The flourishing trend of algebraic research around weak Kleene logics has strongly connected them with the algebraic theory of P\l onka sums (see e.g. \cite{Bonziobook}). Recently, the tools offered by P\l onka sums have fruitfully been extended to the structural analysis of residuated structures, establishing a natural connection with substructural logics (see \cite{Jenei}, \cite{Jipsen}). In line with this trend, we will further extend the application of the method. Not surprisingly, since Bochvar external logic is a linguistic expansion of Bochvar logic, we will show that the construction of the P\l onka sum will play an important role in characterizing the structure of Bochvar algebras.


The paper is organized into five sections: in Section \ref{sec: Bochvar logic}, we present Bochvar external logic as the logic induced by a single matrix, and we recall the axiomatization due to Finn and Grigolia. In Section \ref{sec: 3}, we first introduce the quasivariety of Bochvar algebras and prove some basic facts, including that any Bochvar algebra has an involutive bisemilattice reduct. We then proceed by describing the structure of Bochvar algebras: the main result is a representation theorem in terms of P\l onka sums of Boolean algebras plus some additional operations. Section \ref{sec: sottoquasi e estensioni} is concerned with the study of  the lattice of  subquasivarieties of the quasivariety of Bochvar algebras, which is dually isomorphic to the lattice of extension of Bochvar external logic. We show that there are only three non trivial quasivarieties of Bochvar algebras and we address the problem of (passive) structural completeness for each of them. Finally, in Section \ref{sec: bridge properties}, we show that every quasivariety of Bochvar algebras has the amalgamation property. We conclude the paper with Appendix \ref{appendix}, where we provide a new (quasi)equational basis for the quasivariety of Bochvar algebras. The proposed axiomatization significantly simplifies the traditional one introduced by Finn and Grigolia \cite{FinnGrigolia}.



\section{Bochvar external logic}\label{sec: Bochvar logic}


Kleene's three-valued logics -- introduced by Kleene in his \emph{Introduction to Metamathematics} \cite{Kleene} -- are traditionally divided into two families, depending on the meaning given to the connectives: \emph{strong Kleene} logics -- counting strong Kleene and the logic of paradox -- and \emph{weak Kleene} logics, namely Bochvar logic \cite{Bochvar} and paraconsistent weak Kleene logic (sometimes referred to as Hallden's logic \cite{Hallden}). Kleene logics are traditionally defined over the (algebraic) language of classical logic. However, the intent of one of the first developers of these formalisms, D. Bochvar, was to work within an enriched language allowing to express all classical ``two-valued'' formulas -- which he referred to as \emph{external formulas} -- beside the genuinely ``three-valued'' ones.

The result of this choice is the language $\mathcal{L}\colon\langle \neg,\lor,\land,\Jzero,\Juno,\Jdue,0,1\rangle$ (of type $(1,2,2,1,1,1,0,0)$), which is obtained by enriching the classical language by three unary connectives $\Jzero,\Juno,\Jdue$ (and the constants $0,1$).
The language $\mathcal{L}$ can be referred to as \emph{external language}, in contrast with the traditional language upon which Kleene logics are defined. Let  $\Fm$ refer to the formula algebra over the language $\mathcal{L}$, and to $Fm$ as its universe.

The intended algebraic interpretation of the language $\mathcal{L}$ is traditionally given via the three-elements algebra $\WKt=\pair{\{0,1,\ant\}, \neg,\lor,\land,\Jzero,\Juno,\Jdue,0,1}$ displayed in Figure \ref{fig:WKe}. 
\begin{figure}[h]\
\begin{center}\renewcommand{\arraystretch}{1.25}
\begin{tabular}{>{$}c<{$}|>{$}c<{$}}
   & \lnot  \\[.2ex]
\hline
  1 & 0 \\
  \ant & \ant \\
  0 & 1 \\
\end{tabular}
\qquad
\begin{tabular}{>{$}c<{$}|>{$}c<{$}>{$}c<{$}>{$}c<{$}}
   \lor & 0 & \ant & 1 \\[.2ex]
 \hline
       0 & 0 & \ant & 1 \\
       \ant & \ant & \ant & \ant \\          
       1 & 1 & \ant & 1
\end{tabular}
\qquad
\begin{tabular}{>{$}c<{$}|>{$}c<{$}>{$}c<{$}>{$}c<{$}}
   \land & 0 & \ant & 1 \\[.2ex]
 \hline
     0 & 0 & \ant & 0 \\
     \ant & \ant & \ant & \ant \\          
    1 & 0 & \ant & 1
\end{tabular}
\\
\vspace{10pt}
\begin{tabular}{>{$}c<{$}|>{$}c<{$}}
  \varphi & \Jzero \varphi \\[.2ex]
\hline
  1 & 0 \\
  \ant & 0 \\
  0 & 1 \\
\end{tabular}
\qquad
\begin{tabular}{>{$}c<{$}|>{$}c<{$}}
  \varphi & \Juno\varphi  \\[.2ex]
\hline
  1 & 0 \\
  \ant & 1 \\
  0 & 0 \\
\end{tabular}
\qquad
\begin{tabular}{>{$}c<{$}|>{$}c<{$}}
  \varphi & J_{_2} \varphi \\[.2ex]
\hline
  1 & 1 \\
  \ant & 0 \\
  0 & 0 \\
\end{tabular}

\end{center}
\caption{The algebra $\WKt$.}\label{fig:WKe}
\end{figure}


The value $\ant$ is traditionally read as ``meaningless'' (see e.g. \cite{ferguson2017meaning} and \cite{SmuczFerguson}) due to its infectious behavior.
It is immediate to check that the $\vee,\land$-reduct of $\WKt$ is not a lattice (it is an involutive bisemilattice), as it fails to satisfy absorption, hence the operations $\vee$ and $\land$ induce two (different) partial orders. In the following, we will refer to $\leq$ as the one induced by $\lor$ (i.e. $x\leq y$ iff $x\lor y = y$). With reference to such order, it holds $ 0 < 1 < \ant$.



The language $\mathcal{L}$  allows to define the so-called \emph{external} formulas (see Definition \ref{def: formula esterna}), namely those that are evaluated into $\{0,1\}$ \emph{only} (which is the universe of a Boolean subalgebra of $\WKt$), for any homomorphism $h\colon\Fm\to\WKt$ ($\Ji \varphi$, for any $\varphi\in Fm$ and $k\in \{0,1,2\}$, are examples of external formulas). 

%


\begin{definition}
Bochvar \emph{external} logic $\B$ is the logic induced by the matrix $\pair{\WKt,\{1\}}$. 
\end{definition}
In words, $\B$ is the logic with the only distinguished value $1$.\footnote{The different choice (on the same formula algebra) of the truth set $\{1,\ant\}$ defines the logic $\Ho$ studied by Segerberg \cite{Segerberg65}.}
$\B$ is a linguistic expansion of Bochvar logic $\mathsf{B}$, which is defined by the matrix $\pair{\WK, \{1\}}$, where $\WK$ is the $\Ji$-free reduct of $\WKt$. Since $\B$ is defined by a finite set of finite matrices, it is a finitary logic, in the sense that $\Gamma\vdash_{\B}\varphi$ entails $\Delta\vdash_{\B}\varphi$ for some finite $\Delta\subseteq\Gamma$ (the relation among $\B$ and other three-valued logics can be found in \cite{Ciucci}).

The following technicalities are needed to introduce a Hilbert-style axiomatization of $\B$.

\begin{definition}\label{def: variabili aperte/coperte}
An occurrence of a variable $x$  in a formula $\varphi$ is \emph{open} if it does not fall under the scope of $\Ji$, for every $k\in\{0,1,2\}$. A variable $x$ in $\varphi$ is \emph{covered} if all of its occurrences are not open, namely if  every occurrence of $x$ in $\varphi$ falls under the scope of $\Ji$, for some $k\in\{0,1,2\}$.
 \end{definition}

The intuition behind the notion of external formulas is made precise by the following.

\begin{definition}\label{def: formula esterna}
A formula $\varphi\in Fm$ is called \emph{external} if all its variables are covered. 
\end{definition}


A Hilbert-style axiomatization of $\B$ has been introduced by Finn and Grigolia \cite{FinnGrigolia}. In order to present it, let $$\varphi\equiv\psi\coloneqq\bigwedge_{i=0}^{2}J_{_i}\varphi\leftrightarrow J_{_i}\psi, $$
and $\alpha,\beta,\gamma$ denote external formulas.

\vspace{5pt}
\noindent
\textbf{Axioms}
\begin{itemize}
\item[(A1)] $(\varphi\lor\varphi)\equiv\varphi$;
\item[(A2)] $(\varphi\lor\psi)\equiv(\psi\lor\varphi)$;
\item[(A3)] $((\varphi\lor\psi)\lor \chi)\equiv(\varphi\lor(\psi\lor\chi))$;
\item[(A4)] $(\varphi\land(\psi\lor \chi)\equiv((\varphi\land\psi)\lor(\varphi\land\chi))$;
\item[(A5)] $\neg(\neg \varphi)\equiv\varphi$;
\item[(A6)] $\neg 1\equiv 0$;
\item[(A7)] $\neg( \varphi\lor\psi)\equiv(\neg\varphi\land\neg\psi)$;
\item[(A8)] $0\vee \varphi\equiv\varphi$;
\item[(A9)] $J_{_2}\alpha\equiv\alpha$;
\item[(A10)] $J_{_0}\alpha\equiv\neg\alpha$;
\item[(A11)] $J_{_1}\alpha\equiv 0$;
\item[(A12)] $J_{_i}\neg\varphi\equiv J_{_{2-i}}\varphi$, for any $i\in\{0,1,2\}$;
\item[(A13)] $ J_{_i}\varphi \equiv\neg(J_{_j}\varphi\vee J_{_k}\varphi)$, with $i\neq j\neq k\neq i$;

\item[(A14)] $(J_{_i}\varphi\vee\neg J_{_i}\varphi)\equiv 1$, with $i\in\{0,1,2\}$;

\item[(A15)] $((J_{_i}\varphi\vee J_{_k}\psi)\land J_{_i}\varphi)\equiv J_{_i}\varphi$, with $i,k\in\{0,1,2\}$;
\item[(A16)] $(\varphi\lor J_{_i}\varphi)\equiv \varphi$, with $i\in\{1,2\}$;
\item[(A17)] $J_{_0}(\varphi\lor\psi)\equiv J_{_0}\varphi\land J_{_0}\psi$;
\item[(A18)] $J_{_2}(\varphi\lor\psi)\equiv ( J_{_2}\varphi\land J_{_2}\psi)\vee( J_{_2}\varphi\land J_{_2}\neg\psi)\vee ( J_{_2}\neg \varphi\land J_{_2}\psi) $;

\item[(A19)] $\alpha \rightarrow (\beta \rightarrow \alpha) $; 

\item[(A20)] $( \alpha \rightarrow ( \beta \rightarrow \gamma ) ) \rightarrow (( \alpha \rightarrow \beta) \rightarrow ( \alpha \rightarrow \gamma))$; 
\item[(A21)] $\alpha \wedge \beta \rightarrow \alpha $; 
\item[(A22)] $\alpha \wedge \beta \rightarrow \beta $;
\item[(A23)]$\left( \alpha \rightarrow \beta \right) \rightarrow
\left( \left( \alpha \rightarrow \gamma \right) \rightarrow \left( \alpha
\rightarrow \beta \wedge \gamma \right) \right) $; \item[(A24)]$\alpha \rightarrow \alpha \vee \beta $; 
\item[(A25)]$\beta \rightarrow \alpha \vee \beta $;
\item[(A26)] $\left( \alpha \rightarrow \gamma \right) \rightarrow
\left( \left( \beta \rightarrow \gamma \right) \rightarrow \left( \alpha
\vee \beta \rightarrow \gamma \right) \right) $;

\item[(A27)] $\left( \alpha \rightarrow \beta \right) \rightarrow
\left( \left( \alpha \rightarrow \lnot \beta \right) \rightarrow \lnot
\alpha \right) $; 

\item[(A28)]$\alpha \rightarrow \left( \lnot \alpha
\rightarrow \beta \right) $; 

\item[(A29)] $\lnot \lnot \alpha \rightarrow \alpha $.
\end{itemize}
\vspace{10pt}
\noindent
\textbf{Deductive rule}
\[
\begin{prooftree}
  \Hypo{\varphi}  \Hypo{\varphi\to\psi}
  \Infer[left label = {[MP]}]2 \psi
\end{prooftree}
\]

\vspace{10pt}

Observe that the axiomatization contains a set of axioms (A19-A29), which, together with the rule of \emph{modus ponens}, yields a complete axiomatization for classical logic (relative to external formulas). 


The fact that $\B$ coincides with the logic induced by the above introduced Hilbert-style axiomatization has been proved in \cite{Ignoranzasevera} (Finn and Grigolia \cite[Theorem 3.4]{FinnGrigolia} only proved a weak completeness theorem for $\B$). We will henceforth indicate by $(\Fm, \vdash_{\B})$ both the consequence relation induced by the matrix $\pair{\WKt,\{1\}}$ and the one induced by the above Hilbert-style axiomatization.

The logic $\B$ is algebraizable with the quasivariety of Bochvar algebras ($\mathsf{BCA}$) -- which will be properly introduced in the next section -- as its equivalent algebraic semantics. This means that there exists maps $\tau\colon Fm\to \mathcal{P}(Eq)$, $\rho\colon Eq\to \mathcal{P}(Fm)$ from formulas to sets of equations and from equations to sets of formulas such that
\[\gamma_{1},\dots,\gamma_{n}\vdash_{\mathsf{B}_{e}}\varphi\iff\tau(\gamma_{1}),\dots,\tau(\gamma_{n})\vDash_{\mathsf{BCA}}\tau(\varphi)\]
and
\[\varphi\thickapprox\psi\Dashv\vDash_{\mathsf{BCA}}\tau\rho(\varphi\thickapprox\psi).\]

The above conditions are verified by setting  $\tau(\varphi)\coloneqq \{x\thickapprox 1\}$ and  $\rho(\varphi\thickapprox\psi)\coloneqq\{\varphi\equiv\psi\}$ (see \cite{Ignoranzasevera} for details).  Moreover, Bochvar external logic enjoys the (global) deduction theorem, which we recall here. 
\begin{theorem}[Deduction Theorem]\label{th: deduzione per Be}
$\Gamma, \psi\vdash_{\B}\varphi$ if and only if $\Gamma\vdash_{\B} \Jdue\psi\to\Jdue\varphi$.
\end{theorem}

\section{Bochvar algebras and P\l onka sums}\label{sec: 3}
 
We assume the reader has some familiarity with universal algebra and abstract algebraic logic (standard references are \cite{Be11g,BuSa00} and \cite{Font16}, respectively). In what follows, given a class of algebras $\class{K}$, the usual class-operator symbols $I(\class{K}), S(\class{K}),H(\class{K}),P(\class{K}),P_{u}(\class{K})$ denote the closure of $\class{K}$ under isomorphic copies, subalgebras, homomorphic images, products and ultraproducts. A class of similar algebras $	\class{K}$ is a quasivariety if $\class{K}=ISPP_{u}(K)$. It is a variety if is also closed under homomorphic images or, equivalently, if $\class{K}=HSP(\class{K})$.

The class of \emph{Bochvar algebras}, $\mathsf{BCA}$ for short, is introduced by Finn and Grigolia \cite[pp. 233-234]{FinnGrigolia} as the algebraic counterpart for $\B$. 

\begin{definition}\label{def: algebre di Bochvar}
A Bochvar algebra $\A=\pair{A,\vee,\wedge, \neg, \Jzero,\Juno,\Jdue, 0,1}$ is an algebra of type $\pair{2,2,1,1,1,1,0,0}$ satisfying the following identities and quasi-identities: 
\begin{enumerate}
\item $\varphi\vee \varphi\thickapprox \varphi$; \label{BCA:1}
\item$\varphi\lor \psi \thickapprox \psi \lor \varphi$; \label{BCA:2}
\item $(\varphi\lor \psi)\lor \delta\thickapprox  \varphi\lor(\psi\lor \delta)$; \label{BCA:3}
\item $\varphi\land(\psi\lor \delta)\thickapprox(\varphi\land \psi)\lor(\varphi\land \delta)$; \label{BCA:4}
\item $\neg(\neg \varphi)\thickapprox \varphi$; \label{BCA:5}
\item $\neg 1\thickapprox 0$; \label{BCA:6}
\item $\neg( \varphi\lor \psi)\thickapprox \neg \varphi\land\neg \psi$; \label{BCA:7}
\item $0\vee \varphi\thickapprox \varphi$; \label{BCA:8}

\item $\Jdue\Ji \varphi\thickapprox\Ji \varphi$, for every $k\in\{0,1,2\}$; \label{BCA:9}

\item $\Jzero\Ji \varphi\thickapprox \neg\Ji \varphi$, for every $k\in\{0,1,2\}$; \label{BCA:10}
\item $J_{_1}\Ji \varphi \thickapprox 0$, for every $k\in\{0,1,2\}$; \label{BCA:11}
\item $J_{_k}(\neg \varphi)\thickapprox J_{_{2-k}}\varphi$, for every $k\in\{0,1,2\}$; \label{BCA:12}
\item $J_{_i}\varphi \thickapprox\neg(J_{_j}\varphi\vee J_{_k}\varphi)$, for $i\neq j\neq k\neq i$; \label{BCA:13}

\item $J_{_k}\varphi\vee\neg J_{_k}\varphi \thickapprox 1$, for every $k\in\{0,1,2\}$; \label{BCA:14}

\item $(J_{_i}\varphi\vee J_{_k}\varphi)\land J_{_i}\varphi\thickapprox J_{_i}\varphi$, for $i,k\in\{0,1,2\}$; \label{BCA:15}
\item $ \varphi\lor J_{_k}\varphi \thickapprox \varphi$, for $k\in\{1,2\}$; \label{BCA:16}
\item $J_{_0}(\varphi\lor \psi)\thickapprox J_{_0}\varphi\land J_{_0}\psi$; \label{BCA:17}
\item $J_{_2}(\varphi\lor \psi)\thickapprox ( J_{_2}\varphi\land J_{_2}\psi)\vee( J_{_2}\varphi\land J_{_2}\neg \psi)\vee ( J_{_2}\neg \varphi\land J_{_2}\psi) $; \label{BCA:18}
\item $\Jzero \varphi \thickapprox \Jzero \psi \;\&\; \Juno \varphi \thickapprox \Juno \psi  \;\&\; \Jdue \varphi \thickapprox \Jdue \psi \;\Rightarrow\; \varphi \thickapprox \psi$. \label{BCA:quasi} 
\end{enumerate}
\end{definition}

$\mathsf{BCA}$ forms a quasivariety which is not a variety \cite{FinnGrigolia80, FinnGrigolia}, and it is generated by $\WKt$, i.e. $\class{BCA}=ISP(\alg{\WKt})$. This is true in virtue of \cite[Theorem 3.2.2]{Cz01}, upon noticing that $\mathsf{BCA}$ algebraizes the logic $\B$, which is defined by the single matrix $\pair{\WKt, \{1\}}$. The fact that $\WKt$ generates $\mathsf{BCA}$ was firstly stated by Finn and Grigolia (\cite{FinnGrigolia80, FinnGrigolia}).  
Familiar examples of Bochvar algebras can be obtained by appropriately computing the external functions over a Boolean algebra, as indicates the following example.
\begin{example}\label{example BA BCA}
Let $\A$ be a (non-trivial) Boolean algebra. Setting the functions $\Ji\colon A\to A$, with $k\in \{0,1,2\}$ as $\Jdue = id$, $\Juno = 0$ (the constant function onto $0$) and $\Jzero = \neg $, then $\A = \pair{A,\wedge,\vee, \neg, 0,1, \Jdue, \Juno, \Jzero}$ is a Bochvar algebra.
\end{example}

For this reason, by $\mathbf{B}_{n}$ we will safely denote both the $n$-elements Boolean algebra and its $\mathsf{BCA}$ expansion obtained according to Example \ref{example BA BCA}. Since $\WKt$ generates $\mathsf{BCA}$, a quasi-equation holds in $\WKt$ if and only if it holds in every $\A\in\mathsf{BCA}$. 

The original equational basis for $\class{BCA}$, as provided in Definition \ref{def: algebre di Bochvar}, can be significantly enhanced by reducing the number of axioms and improving their intelligibility\footnote{We thank an anonymous referee for pointing this out.}. It is known that the operations $J_{_{0}},J_{_{1}}$ can be defined as $\Jdue\neg\varphi$ and $\neg(\Jdue\varphi\lor\Jdue\neg\varphi)$, respectively. Thus, Bochvar algebras can be equivalently presented in the restricted language $\langle\vee,\wedge, \neg, \Jdue, 0,1\rangle$, and this is particularly convenient for our next goal, namely to provide a new, simpler quasi-equational basis for $\class{BCA}$.  This is accomplished in the next theorem.
\begin{theorem}\label{th: algebre di Bochvar2}
The following is a quasi-equational basis for $\class{BCA}$. 
\begin{enumerate}
\item $\varphi\vee \varphi\thickapprox \varphi$; \label{BCA:1}
\item$\varphi\lor \psi \thickapprox \psi \lor \varphi$; \label{BCA:2}
\item $(\varphi\lor \psi)\lor \delta\thickapprox  \varphi\lor(\psi\lor \delta)$; \label{BCA:3}
\item $\varphi\land(\psi\lor \delta)\thickapprox(\varphi\land \psi)\lor(\varphi\land \delta)$; \label{BCA:4}
\item $\neg(\neg \varphi)\thickapprox \varphi$; \label{BCA:5}
\item $\neg 1\thickapprox 0$; \label{BCA:6}
\item $\neg( \varphi\lor \psi)\thickapprox \neg \varphi\land\neg \psi$; \label{BCA:7}
\item $0\vee \varphi\thickapprox \varphi$; \label{BCA:8}
\item $\Jzero\Jdue \varphi\thickapprox \neg\Jdue \varphi$; \label{BCA:10}
\item $\Jdue\varphi \thickapprox\neg(\Jzero\varphi\vee \Juno\varphi)$; \label{BCA:13}
\item $\Jdue\varphi\vee\neg \Jdue\varphi \thickapprox 1$; \label{BCA:14}
\item $J_{_2}(\varphi\lor \psi)\thickapprox ( J_{_2}\varphi\land J_{_2}\psi)\vee( J_{_2}\varphi\land J_{_2}\neg \psi)\vee ( J_{_2}\neg \varphi\land J_{_2}\psi) $; \label{BCA:18}
\item $\Jzero \varphi \thickapprox \Jzero \psi \;\&\; \Jdue \varphi \thickapprox \Jdue \psi \;\Rightarrow\; \varphi \thickapprox \psi$.\label{BCA:quasi}
\end{enumerate}
\end{theorem} The proof of the above Theorem \ref{th: algebre di Bochvar2} requires a significant amount of computations, which are included in the Appendix. Notice that, although  $J_{_{0}},J_{_{1}}$ are definable from the remaining operations of $\class{BCA}$, a detailed investigation of the semantic properties of the full language significantly improve the logical and algebraic understanding of $\class{BCA}$: this is why in several subsequent parts of the paper we will explicitly refer to the full stock of operations.

We now introduce the variety of involutive bisemilattices, which plays a key role to understand the structure theory of Bochvar algebras.
\begin{definition}\label{def: IBSL}
An \emph{involutive bisemilattice} is an algebra $\mathbf{B} = \pair{B,\land,\lor,\lnot,0,1}$ of type $(2,2,1,0,0)$ satisfying:
\begin{enumerate}[label=\textbf{I\arabic*}.]
\item $\varphi\lor \varphi\thickapprox \varphi$;
\item $\varphi\lor\psi\thickapprox \psi\lor \varphi$;
\item $\varphi\lor(\psi\lor \delta)\thickapprox(\varphi\lor \psi)\lor \delta$;
\item $\lnot\lnot \varphi\thickapprox \varphi$;
\item $\varphi\land \psi\thickapprox\lnot(\lnot \varphi\lor\lnot \psi)$;
\item $\varphi\land(\lnot \varphi\lor \psi)\thickapprox \varphi\land \psi$; \label{rmp}
\item $0\lor \varphi\thickapprox \varphi$;
\item $1\thickapprox\lnot 0$.
\end{enumerate}
\end{definition}

The class of involutive bisemilattices forms a variety, which we denote by $\IBSL$.  $\IBSL$ is  the so-called \emph{regularization} of the variety of Boolean algebras: this is the variety satisfying all and only the regular identities that hold in Boolean algebras, namely those identities where exactly the same variables occur in both sides of the equality symbol. The variety $\class{IBSL}$ is generated by the three element algebra $\WK$, i.e. $\class{IBSL}=HSP(\alg{WK})$ (see  \cite[Chapter 2]{Bonziobook}, \cite{Bonzio16},\cite{Plonka69}). It is not a direct consequence of Definition \ref{def: algebre di Bochvar} (nor of Theorem \ref{th: algebre di Bochvar2}) that the $\Ji$-free reduct of any Bochvar algebra is an involutive bisemilattice: indeed Definition \ref{def: algebre di Bochvar} implies that such reduct is a De Morgan bisemilattice (the regularization of de Morgan algebras). However, it is immediate to check that the identity \textbf{I6} in Definition \ref{def: IBSL} holds in any Bochvar algebra, as it does in $\WKt$.
Although being the variety generated by the matrix that defines $\mathsf{B}$, $\mathsf{IBSL}$ is not its algebraic counterpart, which rather consists of the proper quasivariety generated by $\WK$. This quasivariety is called \emph{single-fixpoint involutive bisemilattices}, $\mathsf{SIBSL}$ for short\footnote{This quasivariety is firstly investigated in \cite{Bonzio16}, and \cite{paoli2021extensions} contains information on its constant-free formulation.}, as its members are precisely the involutive bisemilattices with at most one fixpoint for negation, namely those containing at most one element $a$ such that $a=\neg a$. 

The examples of Bochvar algebras we have made so far ($\WKt$ and any Boolean algebra) consist of algebras having an $\mathsf{SIBSL}$-reduct; this is actually true for any Bochvar algebra. 
\begin{proposition}\label{prop: each BCA is SIBSL}
 Every Bochvar algebra has a $\mathsf{SIBSL}$-reduct.
 
\end{proposition}\label{prop: ridotto e IBSL}

\begin{proof}
It suffices to check that every valid $\mathsf{SIBSL}$-quasi-equation is also valid in $\mathsf{BCA}$.
To see this, let $\chi$ be a quasi-equation in the language of $\WK$. We have that
\begin{align*}
\mathsf{SIBSL}\vDash\chi&\iff \\
 \WK\vDash\chi&\iff \\
 \WKt\vDash\chi&\iff \\
 \mathsf{BCA}\vDash\chi.
\end{align*}
The first equivalence holds because $\WK$ generates $\mathsf{SIBSL}$, the second one because $\WK$ is the $\langle J_{_0},J_{_1},J_{_2}\rangle$-free reduct of $\WKt$ and the last one because $\WKt$ generates $\mathsf{BCA}$.
\end{proof}

Clearly, since $\mathsf{SIBSL}\subset\mathsf{IBSL}$, every Bochvar algebra has an $\mathsf{IBSL}$-reduct.
Although this fact, it is not the case that any (single-fixpoint) involutive bisemilattice can be turned into a Bochvar algebra: the reasons will be clear in the last part of the section.

The fact that the $\Ji$-free reduct of every Bochvar algebra is an involutive bisemilattice (Proposition \ref{prop: each BCA is SIBSL}) carries to the relevant observation that any such reduct  can be represented as a P\l onka sum of Boolean algebras \cite{Bonzio16}. P\l onka sums are general constructions introduced by the polish mathematician J. P\l onka \cite{Plo67,Plo67a, plonka1984sum} (more comprehensive expositions are \cite{Romanowska92}, \cite{romanowska2002modes}, \cite[Ch. 2]{Bonziobook} ) -- and now going under his name. In brief, the construction consists of ``summing up'' similar algebras, organized into a semilattice direct system and connected via homomorphisms, into a new algebra. In more details, the (semilattice direct) system is formed by a family of similar algebras $\{\A_{i}\}_{i\in I}$ with disjoint universes, such that the index set $I$ forms a lower-bounded semilattice $(I, \vee, i_{0})$ -- we denote by $\leq$ the induced partial order -- and, moreover, is made of a family $\{p_{ij}\}_{i\leq j}$ of homomorphisms $p_{ij}\colon \A_{i}\to \A_j$, defined from the algebra $\A_i$ to the algebra $\A_j$, whenever $i\leq j$, for $i,j \in I$. Such homomorphisms satisfy a further \emph{compatibility property}: $p_{ii}$ is the identity, for every $i\in I$ and $p_{jk}\circ p_{ij} = p_{ik}$, for every $i\leq j\leq k$. 

Given a semilattice direct system of algebras $\pair{\{\A_i\}_{i\in I}, (I, \vee, i_{0}), \{p_{ij}\}_{i\leq j}}$, the P\l onka sum over it is the new algebra $\A=\PL(\A_{i})_{i\in I}$ (of the same similarity type of the algebras $\{\A_i\}_{i\in I}$) whose universe is the union $A=\displaystyle\bigcup_{i\in I} A_i $ and whose generic n-ary operation $g$ is defined as: 
\begin{equation}\label{eq: operazioni}
g^{\A}(a_{1},\dots, a_n)\coloneqq g^{\A_k}(p_{i_{1}k}(a_1),\dots p_{i_{n}k}(a_n)),
\end{equation}
where $k= i_{1}\vee\dots\vee i_{n}$ and $a_1\in A_{i_1},\dots, a_{n}\in A_{i_n}$. If the similarity type contains any constant operation $e$, then $e^{\A} = e^{\A_{i_0}}$.  The algebras $\{\A_i\}_{i\in I}$ are called the \emph{fibers} of the P\l onka sum. A fiber $\alg{A}_{i}$  is \emph{trivial} if its universe is a singleton.   As already stated above, an element  $a$ is a \emph{fixpoint} when $a=\neg a$. Equivalently, using the P\l onka sum representation, a fixpoint can be understood as the universe of a trivial fiber.

With this terminology at hand, a remarkable result states that any member of the variety $\IBSL$ of involutive bisemilattices is isomorphic to the P\l onka sum over a semilattice direct system of Boolean algebras (see \cite{Bonzio16}, \cite{Bonziobook}). In a P\l onka sum of Boolean algebras  $\A=\PL(\A_{i})_{i\in I}$, a key role is played by  the absorption function $\varphi\land(\varphi\lor\psi)$, in the following sense. Given $a,b\in A$, it is possible to check that $a,b\in A_{i}$ for some $i\in I$ if and only if $a\land(a\lor b)=a$ and $b\land(b\lor a)=b$. Moreover, if $a\in A_{i}$ and $i\leq j$, it holds $p_{ij}(a)=a\land(a\lor b)$, for any $b\in A_{j}$. 
In the light of the above observations, given a Bochvar algebra $\A$ we will denote by $\PL(\A_{i})_{i\in I}$  the P\l onka decomposition of its $\mathsf{IBSL}$-reduct.

It shall be clear to the reader that a Bochvar algebra can not be represented as a P\l onka sum (of some class of algebras) in the usual sense (recalled above). The reason is that the operations $J_{_k} $, for any $k\in\{0,1,2\}$, are not computed according to condition \eqref{eq: operazioni}. Indeed, if $\PL(\A_{i})_{i\in I}$ is a P\l onka sum and $a\in A_{i}$, Axiom \eqref{BCA:14} entails that $J_{k}a\in A_{i_0}$ (with $i_0$ the least element of the index set $I$), while condition \eqref{eq: operazioni} requires that $J_{k}a\in A_{i}$.  Nonetheless, the fact that any Bochvar algebra has an $\mathsf{IBSL}$-reduct (Proposition \ref{prop: ridotto e IBSL}) suggests that P\l onka sums are good candidates to provide a representation theorem. Indeed we will rely on the P\l onka sum representation of the $\mathsf{IBSL}$-reduct of a Bochvar algebra to ``reconstruct'' the additional operations $J_{_k} $ and provide a (unique) P\l onka sum decomposition for any Bochvar algebra. 
We begin by studying the behavior of the maps $J_{_k} $ with respect to the $\mathsf{IBSL}$-reduct of a Bochvar algebra.



\begin{lemma}\label{lemma: J nell'algebra sotto}
Let $\A$ be a (non-trivial) Bochvar algebra and $\PL(\A_{i})_{i\in I}$  be the P\l onka decomposition of the $\mathsf{IBSL}$-reduct of $\A$, having $i_{0}$ as the least element in $I$. Then: 
\begin{enumerate}
\item $\Juno^{\A_{i_0}}$ is the constant map onto $0$; 
\item $\Jdue^{\A_{i_0}} = id$;
\item $\Jzero^{\A_{i_0}} a = \neg a$, for any $a\in A_{i_{0}}$.
\end{enumerate}
\end{lemma}

\begin{proof}
It suffices to notice that (1) holds if and only if  \[\mathsf{BCA}\vDash \varphi\land 0\thickapprox 0\Rightarrow J_{_1}\varphi\thickapprox 0\]
 and this quasi-equation clearly holds in $\WKt$. The same applies to (2), (3) with respect to the quasi-equations $\varphi\land 0\thickapprox 0\Rightarrow J_{_2}\varphi\thickapprox \varphi$, $ \varphi\land 0\thickapprox 0\Rightarrow J_{_1}\varphi\thickapprox \neg \varphi$.
\end{proof}

Observe that, when one takes into account the whole algebra $\A$, $\Juno$ in general does not coincide with the constant function $0$, as in $\WKt$ it holds $J_{_1}\ant=1$.

The next result summarizes the key features of the operation $J_{_1}$.
\begin{lemma}\label{lemma: Juno}
The following hold for every $\A\in\mathsf{BCA}$ (with $\PL(\A_{i})_{i\in I}$ the P\l onka decomposition of its $\mathsf{IBSL}$-reduct).
\begin{enumerate}
 \item if $a,b\in A_{i}$, for some $i\in I$, then $J_{_1}a=J_{_1}b$;
 \item   if $a\in A_{i_{0}}$ then $J_{_1}a=0$;
 \item if $a=\neg a$ then $J_{_1}a=1$;
 \item $p_{i_{0}i}(J_{_1}a)=a\land\neg a$, for every $a\in A_{i}$ and $i\in I$. 
\end{enumerate} 
\end{lemma}

\begin{proof}
(1). For every $a,b\in A_{i}$, $a\land(a\lor b)=a$ and $b\land(b\lor a)=a$. Since 
\[
\WKt\vDash \varphi\land(\varphi\lor \psi)\thickapprox\varphi \ \& \ \psi\land(\psi\lor \varphi)\thickapprox\psi\Rightarrow J_{_1}\varphi\thickapprox J_{_1}\psi,
\]
it follows $J_{_1}a=J_{_1}b$.

(2) follows from  (1) in Lemma \ref{lemma: J nell'algebra sotto}. 

(3) holds because $\WKt\vDash \varphi\thickapprox\neg\varphi\Rightarrow J_{_1}\varphi\thickapprox 1$.

 (4). For $a\in A_{i}$, it holds $p_{i_{0}i}(J_{_1}a)=J_{_1}a\land(J_{_1}a\lor a)$ and observe that 
\[\WKt\vDash J_{_1}\varphi\land(J_{_1}\varphi\lor \varphi)=\varphi\land\neg\varphi,\]

so $p_{i_{0}i}(J_{_1}a)=J_{_1}a\land(J_{_1}a\lor a)=a\land\neg a$. 
\end{proof}

\begin{remark}\label{rem: sui J_1}
It follows from Lemma \ref{lemma: Juno} that, for every $a,b\in A_{i}$ (for some $i\in I$), i.e. $a,b$ are elements in the same fiber of the P\l onka sum, $\Juno a = \Juno b $, thus, in particular, $\Juno(a\vee b) = \Juno a \vee\Juno b = \Juno a \wedge\Juno b = \Juno(a\wedge b)$.
\end{remark}
As a notational convention, let us denote by $1_i$ and $0_i$ the top and the bottom elements, respectively, of a generic fiber $\A_i$ in a P\l onka sum of Boolean algebras.

\begin{lemma}\label{lemma: omomorfismi sono suriettivi}
Let $\A$ be a Bochvar algebra with $\PL(\A_{i})_{i\in I}$ the P\l onka decomposition of its $\mathsf{IBSL}$-reduct. Then 
\begin{enumerate}
\item $\PL(\A_{i})_{i\in I}$ has surjective homomorphisms;
\item  for every $i\neq i_{0}$, $p_{i_{0}i}$ is not injective;
\item for every $a\in A$ and every $i\in I$ (with $a\in A_i$), $\Jdue a\in p_{i_0 i}^{-1}(a)$ and $\Jzero a\in p_{i_0 i}^{-1}(\neg a)$;
\item for every $a,b\in A_i$, $\Jdue (a\vee b ) = \Jdue a\vee \Jdue b$.
\end{enumerate}
\end{lemma}
\begin{proof}
Clearly $\WKt\vDash J_{_2}\varphi\land(J_{_2}\varphi\lor\varphi)\thickapprox\varphi$,  and $\WKt\vDash J_{_0}\varphi\land(J_{_0}\varphi\lor\varphi)\thickapprox\neg\varphi$, so, for $a\in A_{i}$, $p_{i_{0}i}(J_{_2}a)=J_{_2}a\land(J_{_2}a\lor a)=a$ and $p_{i_{0}i}(J_{_0}a)=\neg a$. This proves (3) and that $p_{i_{0}i}$ is surjective, for every $i\in I$. Let now $i\leq j$. Since $p_{i_{0}j}=p_{ij}\circ p_{i_{0}i}$ is surjective, also $p_{ij}$ is surjective. This shows that $\PL(\A_{i})_{i\in I}$ has surjective homomorphisms (1). (4) holds as it is equivalent to the quasi-equation
\[\varphi\land(\varphi\lor\psi)\thickapprox\varphi \ \& \ \psi\land(\psi\lor\varphi)\thickapprox\psi\Rightarrow J_{_2}(\varphi\lor\psi)\thickapprox J_{_2}\varphi\lor J_{_2}\psi,\]
which is true in $\WKt$. \\
\noindent
(2) Suppose, by contradiction, that there is $j\in I$ such that $j\neq i_{0}$ and $p_{i_{0} j}$ is an injective homomorphism. Thus, by Lemma \ref{lemma: Juno} and the fact that $J_{_2}(\varphi\land\neg\varphi)\thickapprox 0$  is true in $\WKt$, it holds $\Juno 0_j = 0$ and $\Jdue 0_j = 0$. Moreover, by \eqref{BCA:12}, $\Jzero 0_j = 1$, hence $J_{_i} 0 = J_{_i} 0_{j}$, for every $i\in\{0,1,2\}$. Therefore, by the quasi-equation \eqref{BCA:quasi}, $0 = 0_{j}$, a contradiction.
\end{proof}

\begin{remark}\label{rem: Jdue inversa a destra}
It follows from Lemma \ref{lemma: omomorfismi sono suriettivi} that, for any $i\in I$, $ p_{i_{0}i}\circ\Jdue^{{\A_{i}}} = id_{\A_i} $, namely that $\Jdue$ (restricted on $\A_i$) is the right inverse of the surjective homomorphism $p_{i_{0}i}$.

\end{remark}

\begin{remark}
Note that, in general, it does not hold that $\Jdue (\varphi\vee \psi)\thickapprox \Jdue \varphi \vee \Jdue \psi$ (the identity is falsified in $\WKt$). 
\end{remark}

Recall that, for any non-trivial Boolean algebra $\A = \pair{A,\wedge,\vee, \neg , 0,1 }$ and any $a\in A$, one can turn the interval $[0,a] = \{x\in A\;|\; x\leq a\}$ into a Boolean algebra $\mathbf{[0,a]} = \pair{[0,a], \wedge, \vee, ^{\ast}, 0, a} $, where $x^{\ast} = \neg x\wedge a$. We will refer to  such an algebra as an \emph{interval Boolean algebra}. In the following result we show that any Boolean algebra in the P\l onka sum representation of the $\mathsf{IBSL}$-reduct of a Bochvar algebra is isomorphic to a specific interval Boolean algebra in the lowest fiber.

\begin{proposition}\label{prop: immagine tramite J2}
Let $\A$ be a Bochvar algebra with $\PL(\A_{i})_{i\in I}$ the P\l onka decomposition of its $\mathsf{IBSL}$-reduct. Then, for every $i\in I$, $\Jdue\colon\A_i \to \mathbf{[0,\top]}$ is an isomorphism onto the interval Boolean algebra ${\bf[0,\top]} = \pair{[0,\top], \wedge,\vee, ^{\ast}, 0 , \top}$, where $\top = \Jdue 1_i$, for every $i\neq i_{0}$. 
\end{proposition}
\begin{proof}
$\WKt\models \Jdue(\varphi\land\psi)\thickapprox\Jdue\varphi\land\Jdue\psi$, thus $\Jdue$ preserves the $\land$ operation. By Lemma \ref{lemma: omomorfismi sono suriettivi}-(4) it also preserves $\lor$ when the arguments belong to the same fiber. This implies that $\Jdue(\A_i)$ is a lattice. To see that is bounded, recall that $\Jdue 0_i = 0$ (with $0_i$ the bottom element of $\A_i$); moreover, let $b\in \Jdue(A_{i})$, thus $b = \Jdue a$, for some $a\in A_i$, then $\Jdue a\vee\Jdue 1_i = \Jdue (a\vee 1_i) = \Jdue 1_i$, i.e. $\Jdue a\leq\Jdue 1_i$ thus $\top= \Jdue 1_i$ is the top element of the lattice $\Jdue(\A_i)$. Finally, observe that, for any $a\in A_i$, $\Jdue \neg a = \Jzero a = \neg\Jdue a\wedge\neg\Juno a = \neg\Jdue a\wedge (\Jdue a \vee\Jzero a) = \neg\Jdue a \wedge (\Jdue a\vee \Jdue\neg a) = \neg\Jdue a \wedge (\Jdue (a\vee \neg a)) =\neg\Jdue a \wedge \Jdue 1_i = (\Jdue a)^{\ast} $, i.e. $\Jdue$ preserves also the complementation of the interval algebra $\mathbf [0,\top]$. We have so shown that $\Jdue $ is a boolean homomorphism. Moreover, $\Jdue$ is injective as $p_{i_{0}i}$ is its left-inverse (see Remark \ref{rem: Jdue inversa a destra}). 
To see that $\Jdue$ is also surjective (onto $[0, \top]$), let $a \in [0, \top]$, i.e. $a\leq \Jdue 1_i$. By Lemma \ref{lemma: J nell'algebra sotto}, $\Jdue a = a\leq \Jdue 1_i$ hence $a = \Jdue a\wedge\Jdue 1_i = \Jdue (a\wedge 1_i) = \Jdue (p_{i_{0}i}(a)\wedge 1_i) = \Jdue (p_{i_{0}i}(a))$. This concludes the proof, because $p_{i_{0}i}(a)\in A_{i}$.
\end{proof}

\begin{remark}\label{rem: J2 e quoziente di A0}
Since $p_{i_{0}i}$ is a surjective (but not injective) homomorphism, for any $i\in I$ (and $i\neq i_0$) 
, then $\A_{i_{0}}/Ker(p_{i_{0}i})\cong\A_i$ (with $Ker(p_{i_{0}i}\neq \Delta^{\A_{i_0}}$ for $i\neq i_{0}$), via the isomorphism $f$ mapping $[x]_{Ker(p_{i_{0}i})}\mapsto p_{i_{0}i}(x)$. The proof of Proposition \ref{prop: immagine tramite J2} shows that $\Jdue$ is in fact the inverse of $f$. 
\end{remark}

Combining Proposition \ref{prop: immagine tramite J2} and Remark \ref{rem: J2 e quoziente di A0} we get that, for any $i\in I$ $\mathbf{[0,\top]}\cong \A_{i_{0}}/Ker(p_{i_{0}i})$. In the following we use the interval characterization proved in Proposition \ref{prop: immagine tramite J2} to establish some properties that will be used in the subsequent sections.

\begin{lemma}\label{lemma: J2 di 1}
Let $\A$ be a Bochvar algebra with $\PL(\A_{i})_{i\in I}$ the P\l onka decomposition of its $\mathsf{IBSL}$-reduct. Then: 
\begin{enumerate}
\item If $i < j$, then $\Jdue 1_j < \Jdue 1_i$. In particular, $\Jdue(\A_j) = [0,\Jdue 1_j]\subset [0,\Jdue 1_i] = \Jdue(\A_i)$; 
\item $\Jdue(p_{ij}(a))\leq \Jdue a$, for every $i\leq j$ and $a\in A_{i}$.
\end{enumerate}
\end{lemma}

\begin{proof}
(1). If $1_{i}<1_{j}$ then $1_{i}\land1_{j}=1_{j}\neq 1_{i}$. The following two quasi-equations hold in $\WKt$
\begin{align*}
  \varphi\lor\neg\varphi\leq\psi\lor\neg\psi\Rightarrow J_{_2}(\psi\lor\neg\psi)\leq J_{_2}(\varphi\lor\neg\varphi), \\
  J_{_2}(\varphi\lor\neg\varphi)\thickapprox J_{_2}(\psi\lor\neg\psi)\Rightarrow \varphi\lor\neg\varphi\thickapprox\psi\lor\neg\psi.
\end{align*}
From the former we have that $J_{_2}1_{j}\leq J_{_2}1_{i}$ and, since $1_i\neq 1_j$, from the latter we conclude $J_{_2}1_{j}<J_{_2}1_{i}$. This entails $\Jdue(\A_j) = [0,\Jdue 1_j]\subset [0,\Jdue 1_i] = \Jdue(\A_i)$. \\
\noindent
(2) is equivalent to the equation $J_{_2}(\varphi\land(\varphi\lor\psi))\leq J_{_2}\varphi$, which is true in $\WKt$.
\end{proof}

The following result provides necessary conditions for an $\mathsf{SIBSL}$ to be the reduct of a Bochvar algebra.


\begin{theorem}[P\l onka sum decomposition]\label{th: Plonka sum decomposition}
Let $\A$ be a Bochvar algebra with $\PL(\A_i)_{i\in I}$ the P\l onka sum representation of its $\mathsf{IBSL}$-reduct. Then:
\begin{enumerate}
 \item  all the homomorphisms $\{p_{ij}\}_{i\leq j}$ are surjective and $p_{i_{0}i}$ is not injective for every $i\neq i_{0}$;
 \item for every $i\in I$, there exists an element $a_{i}\in A_{i_{0}}$ such that the restriction $p_{i_{0}i} $ on $\mathbf{[0,a_{i}]}$ is an isomorphism (with inverse $\Jdue$) onto $\A_{i}$, with $a_{i}\neq a_{j}$ for every $i\neq j$; in particular, if $i<j$ then $a_{j}<a_{i}$.
 
\end{enumerate}
Moreover, the decomposition is unique up to isomorphism.
\end{theorem}

\begin{proof}
Let $\A\in\mathsf{BCA}$. By Proposition \ref{prop: ridotto e IBSL}, the $\Ji$-free reduct (with $k\in\{0,1,2\}$) of $\A$ is a single-fixpoint involutive bisemilattice, thus it is isomorphic to a P\l onka sum over a semilattice direct systems $\pair{\{\A_i\}_{i\in I}, (I,\leq), p_{ij}}$ of Boolean algebras (\cite{Bonzio16, Bonziobook}), whose homomorphisms are surjective and $p_{i_{0}i}$ is not injective (by Lemma \ref{lemma: omomorfismi sono suriettivi}). 
Moreover, for every $i\in I$, $\Jdue 1_{i}$ is the element in $\A_{i_{0}}$ such that $\A_{i}\cong\mathbf{[0, \Jdue 1_{i}]}$, by Proposition \ref{prop: immagine tramite J2} (which also ensures the isomorphisms are given by $\Jdue$ and $p_{i_{0}i}$). Lemma \ref{lemma: J2 di 1} ensures that, if $i<j$ then $J_{_2}1_{j}<J_{_2}1_{i}$. Finally, notice that $J_{_2}(\varphi\lor\neg\varphi)\thickapprox J_{_2}(\psi\lor\neg\psi)\Rightarrow \varphi\lor\neg\varphi\thickapprox \psi\lor\neg\psi$ holds in $\WKt$. Therefore, since $i\neq j$ entails $1_{i}\neq 1_{j}$,  we conclude $a_{i}=J_{_2}1_{i}\neq J_{_2}1_{j}=a_{j}$.

We now show that the decomposition is unique up to isomorphism, namely that different choices of the element $\Jdue a\in p_{i_{0}i}^{-1}(a)$ (for any $a\in A_{i}$) on isomorphic $\mathsf{IBSL}$-reducts lead to isomorphic Bochvar algebras. Observe that the decomposition of the $\mathsf{IBSL}$-reduct of $\A$ is unique up to isomorphism (\cite{Bonzio16, Bonziobook}). So, suppose that $\A$ and $\mathbf{B}$ are Bochvar algebras whose $\mathsf{IBSL}$-reducts $\A'$, $\mathbf{B}'$ are isomorphic via an isomorphism $\Phi\colon \A'\to\mathbf{B}'$. We claim that $\A\cong\mathbf{B}$ via $\Phi$. To this end, we want to show that $\Phi$ preserves also the operations $J_{_{k}}$, for any $k\in\{0,1,2\}$. Recall from the theory of P\l onka sums that $\Phi$ preserves the fibers (details can be found in \cite{Loi, SB18}) in the following sense: for any fiber $\A_i$ in the P\l onka sum decomposition of $\A'$, $\Phi(\A_i) \cong \mathbf{B}_{\varphi(i)}$, where $\varphi$ is the isomorphism induced by $\Phi$ on the semilattice of indexes and $\mathbf{B}_{\varphi(i)}$ is a fiber in the P\l onka sum decomposition of $\mathbf{B}' $. 
In particular, the following diagram commutes, for any $i\in I$ (with a slight abuse of notation we indicate $\varphi(i_0)$ with $i_0$ as it is still a lower bound in $\varphi(I)$).
\begin{center}

\begin{tikzpicture}
\draw (-4,0) node {$ \A_{i_0} $};
	\draw (4,0) node {$ \mathbf{B}_{i_0} $};
	
	\draw [line width=0.8pt, ->] (-3.6,0) -- (3.6,0);
	\draw (0,-0.4) node {\begin{footnotesize}$\Phi_{\restriction{\A_{i_{0}}}}$\end{footnotesize}};
	
	\draw [line width=0.8pt, <-] (3.3,3) -- (-3.3,3);
	\draw (0, 3.3) node {\begin{footnotesize}$\Phi_{\restriction{\A_{i}}}$\end{footnotesize}};

	\draw (-4,3) node {$\A_i$}; 
	
	\draw (4,3) node {$\mathbf{B}_{\varphi(i)}$};
		
	\draw [line width=0.8pt, <-] (-4,2.6) -- (-4,0.4);
	\draw (-4.6,1.7) node {\begin{footnotesize}$p_{i_{0}i}$\end{footnotesize}};
	\draw (4.6,1.7) node {\begin{footnotesize}$q_{i_{0}\varphi(i)}$\end{footnotesize}};
	\draw [line width=0.8pt, ->] (3.9,0.4) -- (3.9,2.6);
	
\end{tikzpicture}
\end{center}
\noindent
We claim that $\A\cong\mathbf{B}$ via $\Phi$. Let $a\in A$; in particular, $a\in A_i$, for some $i\in I$. By Lemma \ref{lemma: omomorfismi sono suriettivi}, $\Jdue^{\A} a \in p_{i_{0}i}^{-1}(a)$ and $\Jdue^{\mathbf{B}}\Phi(a)\in q_{i_{0}\varphi(i)}^{-1}(\Phi(a))$, hence, by the commutativity of the above diagram, the fact that $\Jdue$ is the isomorphism between $\A_{i}$ and $ \A_{i_{0}}/Ker(p_{i_{0}i}) $ (Remark \ref{rem: J2 e quoziente di A0}) and that $\Phi_{\restriction{\A_{i}}}$ and $\Phi_{\restriction{\A_{i_{0}}}}$ are isomorphisms) follows that $\Phi(\Jdue^{\A} a) = \Jdue^{\mathbf{B}}\Phi (a)$.
Therefore, we have that $\Phi$ is a homomorphism with respect to $\Jdue$ (and hence with respect of $\Juno$ and $\Jzero$, which can be defined in term of $\Jdue$) and this shows that $\A\cong\mathbf{B}$, namely the P\l onka sum decomposition is unique up to isomorphism. 
\end{proof}

\begin{remark}
 Observe that, as a consequence of Theorem \ref{th: Plonka sum decomposition}, the P\l onka sum decomposition of a Bochvar algebra $\A$ admits no injective homomorphism (excluding the identical homomorphisms $p_{ii}$). To see this, suppose $p_{ij}$ is injective, for some $i<j$. Then $p_{ij}$ is an isomorphism, as it is also a surjective map. Observe that $a_{j}=J_{_2}1_{j}<a_{i}=J_{_2}1_{i}$ and $p_{i_{0}i}\colon\mathbf{[0,a_{i}]}\to\A_{i}$ is also an isomorphism. Therefore, $p_{i_{0}i}(a_{j})\neq 1_{i}=p_{i_{0}i}(a_{i})$ and, by the injectivity of $p_{ij}$, $p_{ij}\circ p_{i_{0}i}(a_{j})=p_{i_{0}j}(a_{j})\neq 1_{j}$, a contradiction.
\end{remark}


We now show that the conditions displayed in the above theorem are also sufficient to equip any $\mathsf{SIBSL}$ with a $\mathsf{BCA}$-structure.

\begin{theorem}\label{thm: converse decomposition}
Let $\A = \pair{A,\wedge,\vee,\neg, 0,1}$ be an involutive bisemilattice whose P\l onka sum representation is such that

\begin{enumerate}
 \item all homomorphisms are surjective and $p_{i_{0}i}$ is not injective for every $i_{0}\neq i\in I$;
 \item for each $i\in I$ there exists an element $a_{i}\in A_{i_{0}}$ such that $p_{i_{0} i}\colon \mathbf{[0},\mathbf{a}_{i}]\to \A_i$ is an isomorphism, with $a_{i}\neq a_{j}$ for $i\neq j$ and, in particular, $a_{j}<a_{i}$ for each $i<j$.
\end{enumerate}
Define, for every $a\in A_{i}$ and $i\in I$:
\begin{itemize}
\item $J_{_2}(a)=p_{i_{0}i}^{-1}(a)\in[0,a_{i}]$; 
\item $\Jzero a\coloneqq \Jdue (\neg a)$; 
\item $\Juno a\coloneqq \neg(\Jdue a\vee\Jdue(\neg a))$.
\end{itemize}
Then $\mathbf{B} = \pair{A,\wedge,\vee,\neg, 0,1, \Jdue, \Juno, \Jzero}$ is a Bochvar algebra. 
\end{theorem}
\begin{proof}
It is immediate to check that assumption (1) implies that $\A\in\mathsf{SIBSL}$. Since $p_{i_{0}i}\colon[0,\mathbf{a}_{i}]\to\A_{i}$ is an isomorphism (with inverse $p_{i_{0}i}^{-1}$) the maps $\Jdue$, $\Juno$ and $\Jzero$ are well defined and naturally extend to the whole algebra $\A$. It can be mechanically checked that $\mathbf{B}$ satisfies all the quasi-equations in Definition \ref{def: algebre di Bochvar}.
\end{proof}

The first part of condition (2) in Theorems \ref{th: Plonka sum decomposition} and \ref{thm: converse decomposition} can be replaced by the assumption that $1/\Ker_{p_{i_{0}i}}$ is a principal filter, for every $i\in I$. Therefore, we obtain the following.

\begin{corollary}\label{cor: corollario su th rappresentazione}
 Let $\A\in\mathsf{SIBSL}$ be with surjective and non injective homomorphisms.
 The following are equivalent:
 
\begin{enumerate}
 \item $\A$ is the reduct of a Bochvar algebra;
 \item for each $i\in I $, $1/\Ker_{p_{i_{0}i}}$ is a principal filter, with generator $a_{i}\in A_{i_{0}}$. Moreover,  if $i\neq j$  then $a_{i}\neq a_{j}$ and $a_{j}<a_{i}$ for each $i<j$;
 \item  for each $i\in I$ there exists an element $a_{i}\in A_{i_{0}}$ such that $p_{i_{0} i}\colon \mathbf{[0},\mathbf{a}_{i}]\to \A_i$ is an isomorphism. Moreover,  if $i\neq j$  then $a_{i}\neq a_{j}$ and $a_{j}<a_{i}$ for each $i<j$.
\end{enumerate}
\end{corollary}
\begin{proof}
We just show $(2)\Leftrightarrow(3)$. Let $1/\Ker_{p_{i_{0}i}}$ be the filter generated by $a_{i}$, for some $a_i \in A_{i_{0}}$. It is routine to check that $p_{i_{0}i}\colon\mathbf{[0},\mathbf{a}_{i}]\to \A_i$ is an isomorphism. Conversely, assume that there is an element $a_i\in A_{i_0}$ such that $p_{i_{0}i}\colon\mathbf{[0},\mathbf{a}_{i}]\to \A_i$ is an isomorphism. Suppose, by contradiction, that $1/\Ker_{p_{i_{0}i}}$ is not principal, i.e. there is an element $b\in A_{i_{0}}$ such that $a_i\nleq b$ and $p_{i_{0}i}(b) = 1_i$. Then $c = b\wedge a_i$ is an element in $[0, a_i]$ such that $p_{i_{0}i}(c) = 1_i $, in contradiction with the fact that $p_{i_{0}i}$ is an isomorphism.
\end{proof}

We conclude this section with an example which empathizes the role of condition (2) in Corollary \ref{cor: corollario su th rappresentazione} and shows an $\mathsf{SIBSL}$ that can not be turned into a Bochvar algebra.

\begin{example}
 Let $\mathcal{P}(\mathbb{Z})$ be the power set Boolean algebra over the integers. This algebra is uncountable and atomic, with $\mathbb{Z}$ as top element. Consider now  the non-principal ideal $I$ containing all the finite subsets of $\mathbb{Z}$. 
 Observe that $\mathcal{P}(\mathbb{Z})/I$ is an atomless Boolean algebra.
Let  $\A$ be the P\l onka sum built over the two fibers $\mathcal{P}(\mathbb{Z}), \mathcal{P}(\mathbb{Z})/I$ and with (the unique non-trival) homomorphism the canonical map  $p\colon \mathcal{P}(\mathbb{Z})\to \mathcal{P}(\mathbb{Z})/I$. Clearly $\A$ satisfies condition (1) of Theorem \ref{thm: converse decomposition} (the homomorphism $p$ is  both surjective and non injective hence $\A\in\mathsf{SIBSL}$). However the filter $\mathbb{Z}/\Ker_{p}$ (corresponding to the ideal $I$)
is not principal, as $I $ is not. 
In the light of condition (2) of Corollary \ref{cor: corollario su th rappresentazione}, $\A$ is a $\mathsf{SIBSL}$ which is not the reduct of any Bochvar algebra.
 
 \end{example}

The results included in the subsequent sections will strongly make use of the P\l onka sum decomposition of a Bochvar algebra provided in Theorem \ref{th: Plonka sum decomposition}. In the exposition of many results we will take for granted some of the details introduced in this section.

\section{On  quasivarieties of Bochvar algebras}\label{sec: sottoquasi e estensioni}

The logic $\B$ is algebraizable with respect to the quasivariety of Bochvar algebras and it is well known that there exists a dual isomorphism between the lattice of extensions of $\B$ and the lattice of subquasivarieties of $\mathsf{BCA}$. In this section we characterize such lattices, proving that they consist of the three-elements chain.  In order to do so, we take advantage of the recent results in \cite{moraschini2020singly}, therefore applying general properties   of passive structurally complete (PSC) quasivarieties. PSC is a weakened variant of structural completeness (see, among others,\cite{raftery2016admissible,campercholi2015structural,rybakov1997admissibility,dzik2016almost}), a notion defined, in its algebraic version, as follows. 
\begin{definition}\label{def: SC}
 A quasivariety  $\mathsf{K}$ is structurally complete (SC) if, for every quasivariety $\mathsf{K}^{\prime}$:
\[
\mathsf{K}^{\prime}\subsetneq\mathsf{K}\;\Rightarrow\;\mathbb{V}(\mathsf{K}^{\prime})\subsetneq\mathbb{V}(\mathsf{K}),
\]
\end{definition}
\noindent
where $\mathbb{V}(\mathsf{K})=HSP(\class{K})$ is the variety generated by $\class{K}$.
Recall that a quasi-identity $ \varphi_{1}\thickapprox\psi_{1}\;\&\dots\&\;\varphi_{n}\thickapprox\psi_{n}\Rightarrow\varphi\thickapprox\psi$ is passive in a quasivariety $\mathsf{K}$ if, for every substitution $h\colon\Fm\to\Fm$, there exists an algebra $\A\in\mathsf{K}$ such that $\A\nvDash h(\varphi_{i})\thickapprox h(\psi_{i})$ for some $1\leq i\leq n$. Put differently, a quasi-identity is passive if it is equivalent to a quasi-identity whose conclusion is $x\thickapprox y$, where $x,y$ are variables that do not appear in the antecedent (see \cite[Sec. 4.3]{agliano2024structural}) for details. We take the following result  from \cite{wronski2009overflow} as a definition of \emph{passive structural completeness} (PSC).
\begin{theorem}[\cite{wronski2009overflow}]\label{def: PSC}
A quasivariety $\mathsf{K}$ is PSC  if every passive quasi-identity over $\mathsf{K}$ is valid in $\mathsf{K}$. 
\end{theorem}

For a quasivariety $\mathsf{K}$, one of the consequences of being PSC amounts to be generated by a single algebra, i.e. to be \emph{singly generated} (see \cite[Section 7]{moraschini2020singly}). Recall that an algebra $\A$ is  a retract of $\mathbf{B}$ if there exist two homomorphisms $\iota\colon\A\to\mathbf{B}$ and $r\colon\mathbf{B}\to\A$   such that $r\circ \iota$ is the identity map on $\A$. This forces $r$ to be surjective and $\iota$ to be injective: we will call $r$ a \emph{retraction}.  $\A$ is a \emph{common retract} of a quasivariety $\mathsf{K}$ if it is a retract of every non-trivial member of $\mathsf{K}$. We denote by $Ret(\mathsf{K},\A)=\{\mathbf{B}\in\mathsf{K}: \A \ \text{is a retract of}\ \mathbf{B}\}$ the members of $\mathsf{K}$ having $\A$ as a retract.

The main tool that will be instrumental for our purposes is the following.

\begin{theorem}\cite[Thm. 7.11]{moraschini2020singly}\label{thm: morasco su psc}
 Let $\mathsf{K}$ be a quasivariety of finite type and $\A\in\mathsf{K}$ a finite $0$-generated algebra.  Then $Ret(\mathsf{K},\A)$ is a maximal PSC subquasivariety of $\mathsf{K}$.\footnote{$\mathsf{K}^{\prime}$ is a maximal PSC subquasivariety of $\mathsf{K}$ when for every PSC quasivariety $\mathsf{K}^{\prime\prime}$,  if $\mathsf{K}^{\prime}\subseteq\mathsf{K}^{\prime\prime}\subseteq \mathsf{K}$ then $\mathsf{K}^{\prime\prime}=\mathsf{K}^{\prime}$. }
\end{theorem}

The following summarizes the properties of the quasivariety $\mathsf{BCA}$ with respect to the above introduced notions.

\begin{proposition}\label{prop: PSC and SC BCA}

The following hold:
\begin{enumerate}[(i)]
 \item $\mathsf{BCA}$ is not PSC;
 \item $\mathbf{B}_{2}$ is the $0$-generated algebra in $\mathsf{BCA}$;
 \item $Ret(\mathsf{BCA},\mathbf{B}_{2})$ is a maximal PSC subquasivariety of  $\mathsf{BCA}$.
\end{enumerate}
\end{proposition}

\begin{proof}
 $(i)$. The quasi-identity
 \begin{equation}\label{NF}\tag{$NF$}
 J_{_1}\varphi\thickapprox 1\Rightarrow \psi\thickapprox 1
\end{equation} 
is passive in $\mathsf{BCA}$. Indeed, for each substitution, the antecedent $\Juno \varphi\thickapprox 1$ is falsified  in every Bochvar algebra containing no trivial algebra in its P\l onka sum decomposition. On the other hand, it is immediate to check that $\mathsf{BCA}\nvDash (NF)$.

 
 $(ii)$ is straightforward, as every Boolean algebra is also a Bochvar algebra (see Example \ref{example BA BCA}), and $\mathbf{B}_{2}$ is the $0$-generated Boolean algebra.
 
 $(iii)$ follows from $(ii)$ and Theorem \ref{thm: morasco su psc}.
\end{proof}

It is easy to check that, in $\mathsf{BCA}$, $(NF)$ is equivalent to the quasi-identity 
$$\varphi\thickapprox\neg \varphi\;\Rightarrow\; \psi\thickapprox \delta.$$

Since $(NF)$ is valid in the quasivariety $Ret(\mathsf{BCA},\mathbf{B}_{2})$, it demands that the underlying involutive bisemilattice of any of its non-trivial members lacks trivial fibers. A stronger fact is established in the following lemma.

\begin{lemma}\label{lem: ret e NBCA}
Let $\A\in\mathsf{BCA}$ be non-trivial. Then $\A\in Ret(\mathsf{BCA},\mathbf{B}_{2})$ if and only if $\A\vDash(NF)$.
\end{lemma}
\begin{proof}
 $(\Rightarrow)$. By Proposition \ref{prop: PSC and SC BCA}, $Ret(\mathsf{BCA}, \mathbf{B}_{2})$ is a (maximal) PSC quasivariety and $(NF)$ is a passive quasi-identity, therefore $Ret(\mathsf{BCA},\mathbf{B}_{2})\vDash(NF)$. \\
\noindent 
 $(\Leftarrow)$. Suppose $\A\vDash(NF)$, therefore its $\mathsf{IBSL}$-reduct lacks trivial fibers. If $\A$ is a Boolean algebra, the conclusion immediately follows. Differently, $\lvert I\rvert\; \geq 2$. Let $i>i_{0}$ and consider the set $X=\{a\in A_{i_{0}}\;|\; p_{i_{0}i} (a) =1_{i}\}$; let $F_0$ be an ultrafilter on $\A_{0}$ containing $X$ (it exists since the filter generated by $X$ is proper, as the P\l onka decomposition lacks trivial algebras). Since all the homomorphisms between fibers of $\A$ are surjective and $\A$ lacks trivial fibers, $F_{j}=p_{i_{0}j}[F_{0}]$ is an ultrafilter on $\A_{j}$, for each $j\in I$ (this readily follows from the fact that $F_0$ is an ultrafilter extending $X$). Set 
 $F=\displaystyle\bigcup_{k\in I}p_{i_{0}k}[F_{0}]$ and let $r\colon\A\to\mathbf{B}_{2}$ be defined, for each $a\in A$, as 
\[
r(a) = \begin{cases}
1, \text{ if } a\in F;  \\
0 \text{ otherwise}.
\end{cases}
\] 
We want to show that $r$ is a retraction.
 We show that $r$ is compatible with $\land$, $ \neg $ and $\Jdue$. Let $a\in A_{i}$, $b\in A_{j}$ and set $k=i\lor j$: we have $r(a\land b)=1\iff a\land b\in F_{k}\iff p_{ik}(a),p_{jk}(b)\in F_{k}$. Suppose, by contradiction, $a\notin F_{i}$. Then $p^{-1}_{i_{0}i}(a)\notin F_{0}$ which entails $p_{i_{0}k}(p^{-1}_{i_{0}i}(a))\notin F_{k}$. However, by the composition property of P\l onka homomorphisms, $p_{i_{0}k}(p^{-1}_{i_{0}i}(a))= p_{ik}(p_{i_{0}i}(p^{-1}_{i_{0}i}(a))=p_{ik}(a)\in F_{k}$, a contradiction. So $a\in F_{i}$. The same argument applies to $b$, and we conclude $b\in F_{j}$. Therefore $r(a)\land r(b)=1=r(a\land b)$. \\
 \noindent
Let $r(a) = 1$, for some $a\in A$, then $a \in F_j$, for some $j\in I$ and, since $F_j$ is an ultrafilter on $\A_j$, $\neg a\not\in F_j$, thus $r(\neg a) = 0 = \neg r(a)$. \\
\noindent 
Finally, we show the compatibility of $r$ with $J_{_2}$. For any $a\in A_{j}$, we have $r(J_{_2}a)=1\iff J_{_2}a\in F_{0}\iff p_{i_{0}i}(J_{_2}a)=a\in F_{i}\iff J_{_2}r(a)=1$.  So $r$ is a surjective homomorphism. Setting $\iota\colon\mathbf{B}_{2}\to \A$ such that $\iota(1)=1, \iota(0)=0$, we have that $r$ is a retraction. This shows $\mathbf{B}_{2}$ is a retract of $\A$, so $\A\in Ret(\mathsf{BCA},\mathbf{B}_{2})$.
\end{proof}

\begin{corollary}\label{cor: NBCA no trivial fibers}
 $Ret(\mathsf{BCA},\mathbf{B}_{2})$ is the quasivariety axiomatized by adding $(NF)$ to the quasi-equational theory of $\mathsf{BCA}$. Moreover, a Bochvar algebra belongs to $Ret(\mathsf{BCA},\mathbf{B}_{2})$ if and only if its P\l onka sum decomposition lacks trivial fibers.
\end{corollary}

 In analogy with the terminology introduced in \cite{paoli2021extensions}, we call the quasivariety $Ret(\mathsf{BCA},\mathbf{B}_{2})$: \emph{nonparaconsistent Bochvar algebras}, $\mathsf{NBCA}$ in brief. Since $\mathsf{NBCA}$ is  a (maximal) PSC subquasivariety of $\mathsf{BCA}$, we know that $\mathsf{NBCA}$, as a quasivariety,  is generated by a single algebra $\alg{A}$, namely $\class{NBCA}=ISPP_{u}(\alg{A})=ISP(\alg{A})$ (\cite[Thm.4.3]{moraschini2020singly}).
  We now introduce an example of Bochvar algebra which will play an important role. 
\begin{example}[$\mathbf{B}_{4}\oplus\mathbf{B}_{2}$]
Let $\mathbf{B}_{4}\oplus\mathbf{B}_{2}$ denote the involutive bisemilattice whose P\l onka sum consists of a system made of the four-element Boolean algebra $\mathbf{B}_{4}$, the two-element one $\mathbf{B}_{2}$, the two-element lattice as index set as in the following diagram (where the arrows stand for the homomorphism $p_{i_{0}i}\colon \mathbf{B}_{4}\to \mathbf{B}_{2}$).

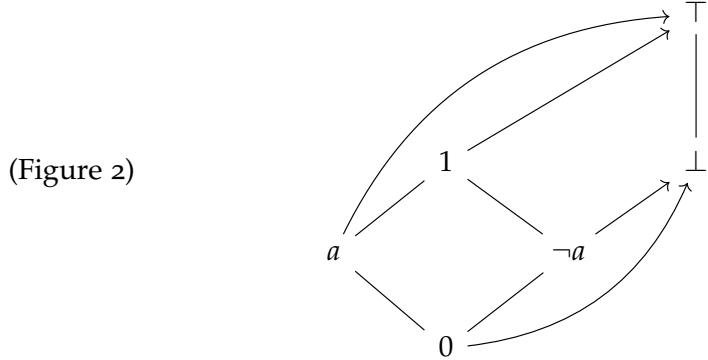
\[
 \begin{tikzcd}[arrows = {dash}]\label{pic: 4+2}
& & & \top & \\
 & & & &   \\
 & 1 \arrow[uurr, ->]& & \perp\arrow[uu, dash] & \\
 a\arrow[ur, dash]\arrow[uuurrr, bend left = 30, ->] & & \neg a \arrow[ul, dash]\arrow[ur, ->]& & \\
   & 0 \arrow[ul, dash]\arrow[ur, dash]\arrow[uurr, ->, bend right = 30]& & &
  \end{tikzcd} \tag{Figure 2}\]
It follows from the structure theory developed in Section \ref{sec: 3} that the unique way to turn $\mathbf{B}_{4}\oplus\mathbf{B}_{2}$ into a Bochvar algebra is by defining $\Jdue \top = a$, $\Jdue\perp = 0$, $\Juno \top = \Juno\perp = \neg a$ and $\Jzero\top = 0$, $\Jzero\perp = a$ (recall that $\Jdue$ is the identity on $\mathbf{B}_4$, $\Juno$ is the constant onto $0$ and $\Jzero$ is negation). With a slight abuse of notation we will indicate this unique Bochvar algebra by $\mathbf{B}_{4}\oplus\mathbf{B}_{2}$, as its $\mathsf{IBSL}$-reduct. \qed
\end{example}

\begin{theorem}\label{thm: nbca generated by 4+2}
The quasivariety $\mathsf{NBCA} $ is generated by $\mathbf{B}_{4}\oplus\mathbf{B}_{2}$. 
\end{theorem}

\begin{proof}
 We show that $\mathsf{NBCA}=ISP(\mathbf{B}_{4}\oplus\mathbf{B}_{2})$. The right to left inclusion is obvious, as subalgebras of direct products of an involutive bisemilattice without trivial fibers preserve the property of lacking trivial fibers.\\
\noindent 
 For the converse, the proof is an adaption of \cite[Thm. 7]{paoli2021extensions}, and we only sketch its main ingredients (leaving the details to the reader). Preliminarily recall that for quasivarieties $\mathsf{K},\mathsf{K}^{\prime}$, $\mathsf{K}\subseteq\mathsf{K}^{\prime}$ if and only if every finitely generated member of $\mathsf{K}$ belongs to $\mathsf{K}^{\prime}$.  Moreover, for algebras $\A,\mathbf{B}$, $\A\in ISP(\mathbf{B})$ if and only if that there exists a family $H\subseteq Hom(\A,\mathbf{B})$ such that $\displaystyle\bigcap_{h\in H}Ker(h)=\Delta^{\A}$.  It is possible to show that, for a finitely generated $\A\in\mathsf{NBCA}$, there exists a family of homomorphisms $H\subseteq Hom(\A,\mathbf{B}_{4}\oplus \mathbf{B}_{2})$ such that $\displaystyle\bigcap_{h\in H}Ker(h)=\Delta^{\A}$. Indeed, $\mathbf{B}_{4}\oplus \mathbf{B}_{2}\cong\WKt\times\mathbf{B}_{2}$ and $\WKt$ generates $\mathsf{BCA}$, so  there exists $H\subseteq Hom(\A,\mathbf{WK})$ such that $\displaystyle\bigcap_{h\in H}Ker(h)=\Delta^{\A}$. Moreover, by Lemma \ref{lem: ret e NBCA}, $\mathbf{B}_{2}$ is a retract of $\A$, so there exists a retraction $r\colon\A\to\mathbf{B}_{2}$. Now, the family $H\times\{r\}$ is a family of homomorphisms from $\A$ to $\mathbf{B}_{4}\oplus \mathbf{B}_{2}$, defined for each $h\in H$ and $a\in A$ by $a\mapsto \langle h(a),g(a)\rangle$. It is easy to check that $\displaystyle\bigcap_{h\in H}Ker \langle h,g\rangle=\Delta^{\A}$, so $\A\in ISP(\mathbf{B}_{4}\oplus \mathbf{B}_{2})$, as desired.
\end{proof}

  Upon noticing that every non-trivial $\mathsf{NBCA}$ lacks trivial subalgebras (this amounts to say that $\mathsf{NBCA}$ is \emph{Koll\'ar}), \cite[Cor 7.8]{moraschini2020singly} ensures that $\mathbf{B}_{2}$ is the unique relatively simple member of $\mathsf{NBCA}$, namely it is the unique algebra in the quasivariety whose lattice of relative congruences is a two-element chain. Moreover, $\mathbf{B}_{4}\oplus\mathbf{B}_{2}$ is a relatively subdirectly irreducible member of $\mathsf{NBCA}$ which is therefore not simple. In other words, $\mathsf{NBCA}$ is not relatively semisimple, unlike $\mathsf{BCA}$.

Observe that any Bochvar algebra satisfies the absorption law $\varphi\thickapprox \varphi\land(\varphi\lor \psi)$ if and only if its involutive bisemilattice reduct is a Boolean algebra. We call $\mathsf{JBA}$ the quasivariety axiomatized by adding the absorption law to the quasi-equational theory of $\mathsf{BCA}$. It is immediate to verify that $\mathsf{JBA}$ and $\mathsf{BA}$ are term equivalent by interpreting the operation $J_{_2}$ as the identity map. The next theorem characterizes the structure of the lattice of non-trivial subquasivarities of $\mathsf{BCA}$, proving that it consists of the following three-elements chain.

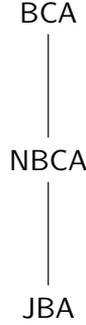
\[
 \begin{tikzcd}[arrows = {dash}]
& \mathsf{BCA}&  \\
 & &    \\
 & \mathsf{NBCA}\arrow[uu]&  \\
  & &  \\
   &\mathsf{JBA}\arrow[uu] & 
  \end{tikzcd}
 \]
\medskip

\begin{theorem}\label{thm: reticolo sottoquasivar}
The only non-trivial subquasivarieties of $\mathsf{BCA}$ are $\mathsf{NBCA}$ and $\mathsf{JBA}$. They form a three element chain $\mathsf{JBA}\subset \mathsf{NBCA}\subset\mathsf{BCA}$.
\end{theorem}
\begin{proof}
We already proved $\mathsf{JBA}\subset\mathsf{NBCA}\subset\mathsf{BCA}$. We only have to show that these are the only non-trivial ones. Suppose that $\mathsf{K}\subseteq\mathsf{BCA}$ and $\mathsf{K}\nsubseteq\mathsf{NBCA}$. Therefore there exists $\A\in\mathsf{K}$ and $\A\notin\mathsf{NBCA}$. This entails $\A\in\mathsf{BCA}$ and $\A$ has a unique trivial fiber with universe $\{a\}$. Clearly $g\colon\WKt\to\A$ mapping $1^{\WKt}\to1^{\A}, 0^{\WKt}\to 0^{\A}, 1/2\to a$ is an embedding. Therefore $\WKt\in S(\A)$, whence $\WKt\in\mathsf{K}$. Since $\WKt$ generates $\mathsf{BCA}$, $\mathsf{K}=\mathsf{BCA}$. 

Suppose now $\mathsf{K}\subseteq\mathsf{NBCA}$ and $\mathsf{K}\nsubseteq\mathsf{JBA}$ and let $\A\in\mathsf{K},\A\notin\mathsf{JBA}$. This entails that the P\l onka sum decompotition of $\A$ has at least two fibers $\A_{i_{0}}, \A_{i}$ $(i_{0}<i)$ and no fiber is trivial. Moreover, by Lemma \ref{lemma: omomorfismi sono suriettivi}, $\A_{i_{0}}$ has cardinality $\geq 4$ (for otherwise $\A_{i}$ would be trivial). Let $h\colon\mathbf{B}_{4}\oplus\mathbf{B}_{2}\to \A$ mapping $1^{\mathbf{B}_{4}\oplus\mathbf{B}_{2}}\to 1^{\A}, 0^{\mathbf{B}_{4}\oplus\mathbf{B}_{2}}\to 0^{\A}$, $1^{\mathbf{B}_{2}}\to 1_{i}, 0^{\mathbf{B}_{2}}\to 0_{i}$,  $a\to J_{_2}(1_{i}), \neg a\to \neg J_{_2}1_{i}$. Clearly $h$ is an embedding, so $\mathbf{B}_{4}\oplus\mathbf{B}_{2}\in S(\A)$, which entails $\mathsf{K}=\mathsf{NBCA}$.
\end{proof}

\begin{corollary}\label{NBCA SC}
The quasivariety $\mathsf{NBCA}$ is structurally complete.
 
\end{corollary}

\begin{proof}
 The only proper subquasivariety  of $\mathsf{NBCA}$ is the variety $\mathsf{JBA}$, so $\mathbb{V}(\mathsf{JBA})=\mathsf{JBA}\subsetneq \mathsf{NBCA}\subsetneq\mathbb{V}(\mathsf{NBCA})$.
\end{proof}
 
 Moreover, from the fact that $\mathsf{BCA}$ is not SC we can infer the following.
 
  \begin{corollary}\label{NBCA BCA generates same var}
 $\mathbb{V}(\mathsf{NBCA})=\mathbb{V}(\mathsf{BCA})$.  
\end{corollary}

 Let now switch our attention to the logical setting, relying on the bridge results connecting an algebraizable logic (and its extensions) with its algebraic counterpart(s). Let $\mathsf{N}\B$ be the logic obtained by adding to $ \B$ the rule
  \begin{equation*}
J_{_1}\varphi\vdash \psi \tag{$EFJ$}.
 \end{equation*} 
 This logic is a proper extension of $\B$, as $(EFJ)$ is the logical pre-image of $(NF)$ via the transformer formula-equations transformer $\tau$ and $(NF)$ is not valid in  $\B$ (a counterexample is easily found in $\WKt$). In the light of Theorem \ref{thm: reticolo sottoquasivar} we obtain the following.
 
 \begin{corollary}\label{cor: estensioni Be}
 $\NBe$ is complete with respect to the matrix $\langle\mathbf{B}_{4}\oplus\mathbf{B}_{2},\{1\}\rangle$. Moreover, the only non-trivial extensions of $\B$ are $\NBe$ and $\mathsf{CL}$.
 
\end{corollary}

In a logical perspective, the notions of PSC and SC have been investigated in several contribuitions, such as \cite{bergman1988structural,raftery2016admissible,rybakov1997admissibility,wronski2009overflow}. For a logic $\vdash$, being SC amounts to the fact that each admissible rule is derivable in $\vdash$. In other words, $\vdash$ is SC if for every rule (R) of the form $\langle\Gamma,\psi\rangle$:

\[(\vdash\varphi\iff\vdash_{R}\varphi)\Rightarrow \Gamma\vdash\psi\]

where $\vdash_{R}$ is the extension of $\vdash$ obtained by adding (R) to $\vdash$. Clearly, the converse implication in the above display is always true.
A \emph{passive} rule is of the form $\langle\Gamma, y\rangle$, where no member of $\Gamma$ contains occurrences of $y$, namely $y\notin\Var(\Gamma)$. Accordingly, we say that a logic  $\vdash$ is PSC if every passive, admissible rule is derivable.\footnote{In the context of finitary algebraizable logics, the fact that the logical definition of PSC is the ``right''  translation of the algebraic one can be inferred by comparing \cite[Cor.3.3]{moraschini2020singly} and \cite[Thm. 7.5]{raftery2016admissible}.}

The following corollary emphasizes  the logical meaning of the previous results on the subqusivarieties of $\mathsf{BCA}$.
\begin{corollary}\label{rem: SC delle logiche}
$\B$ is not PSC, while $\NBe$ is SC. 
\end{corollary}

\begin{proof}
In order to prove the first statement we show $(EFJ)$ is passive, admissible and non derivable in $\B$. That  $(EFJ)$ is passive and non derivable in $\B$ is clear. Let now $\varphi$ be a theorem of $\NBe$, and remind it is the logic obtained adding $(EFJ)$ to $\B$. Suppose $\varphi$ is not a theorem of $\B$. Then $\mathsf{BCA}\not\vDash\varphi\thickapprox 1$, and $\mathsf{NBCA}\vDash\varphi\thickapprox 1$. However, this contradicts Corollary \ref{NBCA BCA generates same var}. So, $\B$ is not PSC. 

That  $\NBe$ is SC follows straightforwardly upon noticing it has $\mathsf{CL}$ as unique proper (non-trivial) extension, and that $\varphi\lor\neg \varphi$ is not a theorem of $\NBe$.
\end{proof}

\section{Amalgamation in quasivarieties of Bochvar algebras}\label{sec: bridge properties}

In the context of algebraizable logics, several logical properties can be established by means of the so-called bridge theorems, whose general form is
\[
\text{a logic } \logic{L} \text{ has the property } P\iff \class{K} \text{ has the property } Q, 
\] 
where the quasivariety $\class{K}$ is the equivalent algebraic semantics of $\logic{L}$. A valid instance of the above equivalence can be obtained by replacing $P$ with ``Deduction theorem'' and $Q$ with ``Equationally definable principal relative congruences'' (see \cite[Thm.Q.9.3]{Cz01}).

In the light of the results of Section \ref{sec: sottoquasi e estensioni}, $\mathsf{BCA}$ and $\mathsf{NBCA}$  are the only interesting quasivarieties of Bochvar algebras. In this section we show that they enjoy the amalgamation property (AP) or, equivalently, that their associated logics enjoy the Craig interpolation property.
The strategy for proving (AP) for $\mathsf{BCA}$ consists  in providing a sufficient condition implying (AP), established in \cite[Theorem 9]{MetMonTsi} for varieties, and that naturally extends to quasivarieties (Theorem \ref{th: amalgamation for Q}).

Recall that a \emph{V-formation} (see Figure \ref{amalgamo}) is a $5$-tuple $\left(
\mathbf{A},\mathbf{B},\C,i,j\right)  $ such that
$\mathbf{A,B,C}$ are similar algebras, and
$i\colon\mathbf{A\rightarrow B},j\colon\mathbf{A\rightarrow C}$ are embeddings.
A class $\mathsf{K}$ of similar algebras is said to have the
\emph{amalgamation property} if for every V-formation with
$\mathbf{A},\mathbf{B},\C\in\mathsf{K}$ there
exists an algebra $\mathbf{D}\in\mathsf{K}$ and embeddings $h\colon\mathbf{B}%
\mathbf{\rightarrow D},k\colon\C\mathbf{\rightarrow D}$ such that
$k\circ j=h\circ i$. In such a case, we also say that $(\D, h, k)$ is an \emph{amalgam} of the V-formation $\left(  \mathbf{A},\mathbf{B}%
,\C,i,j\right)  $. 

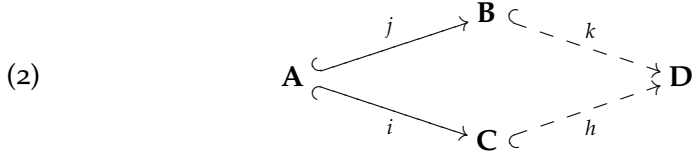
\begin{figure}[h]
\begin{equation}
\vcenter{\xymatrix@R10pt{
                      & & \mathbf{B} \; \ar @{^(-->} [rrd]  ^ {k} & &       \\
 \A \; \ar @{^(->} [rru] ^{j} \ar @{_(->} [rrd] _ {i}& &       & & \D  \\
                      & & \mathbf{C} \;  \ar @{_(-->} [rru] _{h} & &     
}}
\end{equation}
 \caption{A generic amalgamation schema}\label{amalgamo}
\end{figure}

The following lemma is originally due to Gr\"atzer \cite{Gratzer} (it can be also found in \cite{MetMonTsi}), while the subsequent theorem is the obvious adaptation to quasivarieties of a theorem by Metcalfe, Montagna and Tsinakis \cite[Theorem 9]{MetMonTsi}. We insert the proofs for the completeness of the exposition.

\begin{lemma}[\cite{Gratzer}]\label{lemma: amalgamation Q}
Let $\mathsf{Q}$ be a quasivariety. The following are equivalent:
\begin{enumerate}
\item $\mathsf{Q}$ has (AP); 
\item for every $V$-formation $(\A,\mathbf{B},\C, i,j)$ and elements $x\neq y\in B$ ($x\neq y\in C$, respectiely) there exists $\D_{xy}\in \mathsf{Q}$ and homomorphisms $h_{xy}\colon \mathbf{B}\to\D_{xy}$ and $k_{xy}\colon\C\to\D_{xy}$ such that $h_{xy}(x)\neq h_{xy}(y)$ ($k_{xy}(x)\neq k_{xy}(y)$, respectively) and $h\circ i = k\circ j$.
\end{enumerate}
\end{lemma}
\begin{proof}
$(1)\Rightarrow (2)$ is obvious. \\
\noindent
$(2)\Rightarrow (1)$. Let $(\A,\mathbf{B},\C, i,j)$ be a $V$-formation in $\mathsf{Q}$. Define $\D=\displaystyle\prod_{x\neq y\in B} \D_{xy}$. By assumption, for every $x\neq y\in B$ there exist $h_{xy}\colon \mathbf{B}\to \D_{xy}$ and $k_{xy}\colon\C\to\D_{xy}$ s. t. $h(x)\neq h(y)$ and $h\circ i = k\circ j$. By the universal property of the product, $\D$ and the homomorphisms $h$ and $k$, where $\pi_{xy}\circ h = h_{xy}$ and $\pi_{xy}\circ k = k_{xy}$ (with $\pi\colon \D\to\D_{xy}$ the projection) is the amalgam.
\end{proof}

The following provides a sufficient condition for a quasivariety to have the (AP) and it reduces somehow the search for an amalgam to a subclass of a quasivariety.  As a notational convention, by $Co_{\mathsf{K}}^{\A}$ we denote the lattice of $\class{K}$-congruences on an algebra $\A$, namely the congruences $\theta$ on $\alg{A}$ such that $\alg{A}/\theta\in\class{K}$. Let $\{\theta_{i}\}_{i\in I}$ be a family of $\class{K}$-congruences on an algebra $\alg{A}\in\class{K}$. We say that $\alg{A}$ is subdirectly irreducible relative to $\class{K}$, or just relatively subdirectly irreducible, when $\bigwedge_{i\in I}\theta_{i}=\Delta^{\alg{A}}$ entails $\theta_{i}=\Delta^{\alg{A}}$ for some $i\in I$. Moreover, given a quasivariety $\class{K}$, by  $\mathsf{K}_{RSI}$ we indicate the class of relatively subdirectly irreducible members of $\mathsf{K}$. If $\class{K}$ is a variety, we simply write $\class{K}_{SI}$.


\begin{theorem}[essentially \cite{MetMonTsi}]\label{th: amalgamation for Q}
Let $\mathsf{K}$ be a subclass of a quasivariety $\mathsf{Q}$ satisfying the following properties: 
\begin{enumerate}
\item $\mathsf{Q}_{RSI}\subseteq\mathsf{K}$; 
\item $\mathsf{K}$ is closed under $\mathsf{I}$ and $\mathsf{S}$; 
\item for every algebras $\A,\mathbf{B}\in\mathsf{Q}$ such that $\A\leq\mathbf{B}$ and every $\theta\in Co_{\mathsf{K}}^{\A}$ such that $\A/\theta\in \mathsf{K}$ there exists $\Phi\in Co_{\mathsf{K}}^{\mathbf{B}}$ extending $\theta$ with respect to $\mathsf{K}$, i.e. $\mathcal{B}/\Phi\in\mathsf{K}$ and $\Phi\cap A^2 = \theta$; 
\item every $V$-formation of algebras in $\mathsf{K}$ has an amalgam in $\mathsf{Q}$. 
\end{enumerate}
Then $\mathsf{Q}$ has the (AP).
\end{theorem}
\begin{proof}
We show that $\mathsf{Q}$ satisfies the condition (2) in Lemma \ref{lemma: amalgamation Q}. Let $(\A,\mathbf{B},\C, i, j)$ be a $V$-formation in $\mathsf{Q}$ and $x\neq y\in B$. By Zorn lemma, it is possible to find a relative congruence $\Psi$ of $\mathbf{B}$ maximal with respect to the property $(x,y)\not\in\Psi$. Let $\theta = \Psi\cap A^{2}$, and define the map $f\colon\A/\theta\to\mathbf{B}/\Psi$, $[a]_\theta\mapsto f([a]_{\theta})\coloneqq[a]_{\Psi}$. Observe that $f$ is an injective homomorphism. Indeed, for $[a]_{\theta}\neq [b]_{\theta}$, i.e. $(a,b)\not\in\theta$, hence $(a,b)\not\in\Psi$ ($\Psi$ is maximal with respect to this property), i.e. $[a]_{\Psi}\neq [b]_{\Psi}$. $\mathbf{B}/\Psi\in \mathsf{Q}_{RSI}$ (since $\Psi$ is completely meet-irreducible), thus, by hypothesis (1), $\mathbf{B}/\Psi\in \mathsf{K}$; $\mathbf{A}/\theta\leq \mathbf{B}/\Psi$, hence $\mathbf{A}/\theta\in\mathsf{K}$ (by hyp. (2)). Since $\A\leq\C$ (upon identifying $\A$ with $j(\A)$) and $\theta\in Co_{\mathsf{K}}^{\A}$, by hyp (3), there exists $\Phi\in Co_{\mathsf{K}}^{\C}$ s.t. $\Phi\cap A^2 = \theta$ and $\C/\Phi\in\mathsf{K}$. The map $g\colon\mathbf{A}/\theta\to\C/\Phi $ defined as $[a]_{\theta}\mapsto g([a]_{\theta})\coloneqq [a]_{\Phi}$ is an injective homomorphism. Therefore $(\mathbf{A}/\theta, \mathbf{B}/\Psi, \C/\Phi, f, g )$ is a $V$-formation of algebras in $\mathsf{K}$. By hyp (4), there exists an amalgam $(h,k,\D)$ in $\mathsf{Q}$. Define the homomorphisms $h'\colon\mathbf{B}\to\D$ and $k'\colon\C\to\D$ as $h' = h\circ\pi_{\Psi}$ and $k' = k\circ\pi_{\Phi}$ ($\pi_{\Psi}$ and $\pi_{\Phi}$ the projections onto the quotients $\mathbf{B}/\Psi$ and $\C/\Phi$, resp.). Observe that $h'(x)\neq h'(y)$ and $h'\circ i = k'\circ j$. Indeed, $h'(x) = h(\pi_{\Psi}(x))= h([x]_{\Psi}) \neq h([y]_{\Psi}) = h(\pi_{\Psi}(y)) = h'(y)$ (where we have used the injectivity of $h$ and the fact that $[x]_{\Psi}\neq[y]_{\Psi}$). As for the latter, let $a\in A$, $h'\circ i(a) = h(\pi_{\Psi}(i(a))) = h([i(a)]_{\Psi}) = k ([j(a)]_{\Phi}) = k(\pi_{\Phi}(j(a))) = k'(j(a)) = k'\circ j(a) $ (where we have used the fact that $(\D, f,g)$ is an amalgam). Finally, by Lemma \ref{lemma: amalgamation Q}, we conclude that $\mathsf{Q}$ has the (AP).
\end{proof}

\begin{remark}\label{rem: relative CEP}
$\B$ is a finitary logic with a Deduction  Theorem (Theorem \ref{th: deduzione per Be}): this is a stronger property than the local deduction, which  implies that the logic enjoys the filter extension property (\cite[Thm. 2.3.5]{Cz01}). This translates into the relative congruence extension property (by the algebraizability of $\B$, the lattice of logical filters is dually isomorphic to that of the relative congruences). 
\end{remark}



\begin{theorem}\label{teo: AP per BCA}
$\mathsf{BCA}$ has the Amalgamation Property (AP).
\end{theorem}
\begin{proof}
We show that $\mathsf{K} = \mathsf{BCA}_{RSI} = \{\WK^{e},\mathbf{B}_{2}\} $ satisfies the assumptions (1)-(4) of Theorem \ref{th: amalgamation for Q}. (1), (2) and (4) are immediate. As concerns (3): suppose that $\A,\mathbf{B}\in\mathsf{BCA}$ with $\A\leq\mathbf{B}$, $\theta\in Co^{\A}_{\mathsf{K}}$ and $\A_{/\theta}\in \mathsf{K}$. $\mathbf{B}$ decomposes into a P\l onka sum $\PL(\mathbf{B}_{i})_{i\in I}$ and, since $\A\leq\mathbf{B}$, and $\mathsf{S}(\PL(\mathbf{B}_i))\subseteq \PL(\mathsf{S}(\mathbf{B}_{i}))$, then $\A$ decomposes into a P\l onka sum $\PL(\A_{j})$ of subalgebras of $\mathbf{B}_{i}$, over a semilattice of indexes $J\leq I$, thus, in particular, $i_{0}\in J$. Observe that, for every $i\in I$, $\theta_i = \theta\cap A_{i}^{2}$ is a (Boolean) congruence on $\A_i$. The hypothesis that $\A_{/\theta}\in \mathsf{K}$ implies that $\theta_{i_0}$ is a maximal congruence on $\A_{i_0}$ (a congruence corresponding to a maximal ideal). Since $\mathsf{BCA}$ has the relative congruence extension property and $\A\leq \mathbf{B}$ then there exists a relative congruence $\Psi$ on $\mathbf{B}$ extending $\theta$ ($\Psi\cap A^2 = \theta$). $\Psi_{i_0} = \Psi\cap B_{i_0}$ is a (Boolean) congruence on $\mathbf{B}_{i_0}$; let $\Phi_{i_0}$ (one of) its maximal extension on $\mathbf{B}_{i_0}$ and $\Phi$ the congruence on $\mathbf{B}$ defined as follows: 
$ (x,y)\in\Phi $ iff $(\Jdue x,\Jdue y)\in\Phi_{i_0} $. It is immediate to check that $\Phi\in Co_{\mathsf{K}}^{\mathbf{B}}$ and $\mathbf{B}/\Phi\in\mathsf{K}$. Finally, it also holds that $\Phi\cap A^2 = \theta$: $\theta\subseteq \Phi\cap A^2$ follows by construction. On the other hand, let  $a,b\in A$ (with $a\in A_i$ and $b\in B_j$) and $(a,b)\in\Phi$, i.e. $(\Jdue a,\Jdue b)\in\Phi_{i_0}$, hence $(\Jdue a,\Jdue b)\in\theta_{i_0}$ (by construction), thus $\Jdue a,\Jdue b \in [1]_{\theta}$ or $\Jdue a,\Jdue b \in [0]_{\theta}$. Suppose $\Jdue a,\Jdue b \in [1]_{\theta}$ (the other case is analogous), therefore $a = p_{i_{0}i}(\Jdue a)\in [1_i]_{\theta}$ and $b = p_{i_{0}j}(\Jdue b)\in [1_j]_{\theta}$. The assumption that $\A/\theta\in\mathsf{K}$ implies that $[1_i]_\theta = [1_j]_{\theta}$, from which $(a,b)\in\theta$.
\end{proof}

We conclude this section by proving that also the  quasivariety $\mathsf{NBCA}$ has the (AP). Before proceeding further, it is worth noticing that we cannot apply the same strategy used in the case of $\mathsf{BCA}$. This is a consequence of the following remark, which also proves that the logic $\mathsf{NB}_{e}$ does not have a local deduction theorem.

  \begin{remark}
Observe that $\mathsf{NBCA}_{RSI}\subseteq IS(\mathbf{B}_{4}\oplus\mathbf{B}_{2})$, because $ISP(\mathbf{B}_{4}\oplus\mathbf{B}_{2})=IP_{S}S(\mathbf{B}_{4}\oplus\mathbf{B}_{2})$. So, the only relatively subdirectly irreducible members of $\mathsf{NBCA}$ are $\mathbf{B}_{4}\oplus\mathbf{B}_{2}$ and $\mathbf{B}_{2}$. The class $K=\{\mathbf{B}_{4}\oplus\mathbf{B}_{2},\mathbf{B}_{2}\}$ does not satisfy condition (3) of Theorem \ref{th: amalgamation for Q}, consider  the algebra $\mathbf{B}_{4}\oplus\mathbf{B}_{2}$ depicted in \ref{pic: 4+2}, where $J_{_1}(\bot)=\neg a$. Observe that $\mathbf{B}_{4}\leq\mathbf{B}_{4}\oplus\mathbf{B}_{2}$ and  consider $\theta=Cg^{\mathbf{B}_{4}}_{\mathsf{NBCA}}(1,\neg a)$. Clearly $\mathbf{B}_{4}/\theta\cong\mathbf{B}_{2}$ but, for each $\mathsf{NBCA}$-congruence $\Phi$ on $\mathbf{B}_{4}\oplus\mathbf{B}_{2}$,  if $(\neg a,1)\in\Phi$ then $\Phi=\nabla\neq\theta\cap\mathbf{B}_{4}^{2}$. This shows that $\mathsf{NBCA}$ fails the relative congruence extension property or, equivalently, that $\mathsf{NB}_{e}$ fails to have a local deduction theorem.
\end{remark}

Nonetheless, (AP) holds for $\mathsf{NBCA}$, as shown in the following.

\begin{theorem}\label{thm: ap nbca}
 $\mathsf{NBCA}$ has the amalgamation property.
\end{theorem}
\begin{proof}
 Let $(\A,\mathbf{B},\C, f, g)$ be a $V$-formation in $\mathsf{NBCA}$. By theorem \ref{teo: AP per BCA}, there exists an amalgam $(h,k,\mathbf{D})$ with $\D\in\mathsf{BCA}$. If $\D$ has no trivial fibers, then $(h,k,\mathbf{D})$ is also an amalgam in $\mathsf{NBCA}$. Otherwise, let $u\in I$ be the index of the trivial fiber $\mathbf{D}_{u}$ with universe $\{u\}$, where $I$ is the underlying semilattice of $\mathbf{D}$ and homomorphisms $p_{ij}$ for every $i\leq j$. Observe that $k(b)\neq u$ and $h(c)\neq u$, for each $b\in B$, $c\in C$. Consider an ultrafilter $F_{0}$ over $\mathbf{D}_{i_0}$ (the lowest fiber in $\D$) and $F_{i}=p_{i_{0}i}[F_{0}]$, for each $i_0\leq i$ and set $F=\displaystyle\bigcup_{i<u} F_{i}$. Each $F_{i}$ is an ultrafilter over the algebra $\mathbf{D}_{i}$ (since homomorphisms are surjective).  Consider the algebra $\mathbf{D}^{\prime}=\mathbf{D}\times \mathbf{B}_{2}$, which shares with $\mathbf{D}$ the semilattice structure $I$  and whose homomorphisms are denoted by $q_{ij}$ for $i\leq j$. Observe that this algebra does not contain trivial fibers as $u$ is the top element of $I$ and $\mathbf{D}^{\prime}_{u}$ is the two-elements Boolean algebra with universe $\{\langle u,\top\rangle$, $\langle u,\bot\rangle\}$. This entails that $\mathbf{D}^{\prime}\in\mathsf{NBCA}$. Define the map $k'\colon\mathbf{B}\to\mathbf{D}^{\prime}$ such that, for any $b\in B$:
 
\[
k^{\prime}(b) = \begin{cases}
\langle k(b),\top\rangle \text{, if } k(b)\in F, \\
\langle k(b),\bot\rangle \text{ otherwise.}
\end{cases}
\] 
\noindent
Similarly, consider $h^{\prime}\colon\mathbf{C}\to\mathbf{D}^{\prime}$ defined by the same rule when applied to elements in $C$. We show that $k^{\prime} $ is an embedding. Clearly the map in injective, as $k$ is. It is also clear that $k^{\prime}$ preserves the Boolean operations. Moreover, for $b\in B$, 
\begin{align*}
 k^{\prime}(J_{_2}b)=\langle k(J_{_2}b),\top\rangle &\iff \\
  k(J_{_2}b)=J_{_2}k(b)\in F &\iff\\
   k(b)\in F & \iff\\
 J_{_2}k^{\prime}(b)=J_{_2}\langle(k(b),\top\rangle=
 \langle J_{_2}k(b),J_{_2}\top\rangle&=\langle k(J_{_2}b),\top\rangle,
 \end{align*}
 where the second equivalence is justified because, for any $i\in I$ and $d_{i}\in D_{i}$, $J_{_2}d_{i}\in p^{-1}_{0i}(d_{i})$. This, together with an analogous argument applied to $h^{\prime}$, show that $k^{\prime}, h^{\prime}$ are embeddings. In order to conclude the proof, recall that for $a\in A$, $k\circ f(a)=h\circ g(a)$, so the following equivalences hold:
 
\begin{align*}
 k^{\prime}\circ f(a)=\langle k\circ f(a),\top\rangle \iff&\\
 k\circ f(a)\in F\iff&\\
 g^{\prime\prime}\circ g(a)=\langle g^{\prime}\circ g(a),\top\rangle=\langle k\circ f(a),\top\rangle.
\end{align*}
\noindent
This proves $k^{\prime}\circ f(a)=h^{\prime}\circ g(a)$, as desired.
 \end{proof}
 \newpage
 
\appendix\section{}\label{appendix}
 
 This appendix is devoted to the proof of Theorem \ref{th: algebre di Bochvar2}.  Let us denote by $\class{BCA}_{2}$ the quasivariety axiomatized by
(1)-(13) in Theorem \ref{th: algebre di Bochvar2}.
\begin{lemma}\label{lemma: aritmetica 1}
The following identities hold in $\class{BCA}_2$.
\begin{enumerate}
\item $1 \wedge  \varphi \approx \varphi$.
\item$\Ji \varphi\vee \neg\Ji \varphi\approx 1 $, $\forall k\in\{0,1,2\}$.
\item $\Ji \varphi\wedge \neg\Ji \varphi\approx 0 $, $\forall k\in\{0,1,2\}$.
\item $\Jdue(\varphi\vee\neg\varphi) \approx \Jdue\varphi\vee\Jdue\neg\varphi$.
\item$\Jdue\Ji \varphi\thickapprox\Ji \varphi$, for every $k\in\{0,1,2\}$.
\item $\Jzero\Ji \varphi\thickapprox \neg\Ji \varphi$, for every $k\in\{0,1,2\}$.
\item $J_{_i}\varphi \thickapprox\neg(J_{_j}\varphi\vee J_{_k}\varphi)$, for $i\neq j\neq k\neq i$.
\item $J_{_k}\varphi\vee\neg J_{_k}\varphi \thickapprox 1$, for every $k\in\{0,1,2\}$.
\item$((J_{_i}\varphi\vee J_{_k}\varphi)\land J_{_i}\varphi)\approx J_{_i}\varphi$, for $i,k\in\{0,1,2\}.$
\end{enumerate}
\end{lemma}
\begin{proof}
Let $\A\in\class{BCA}_2$ and $a,b\in A$. \\
\noindent
(1) $1\wedge a = \neg(\neg 1\vee \neg a) = \neg(0\vee \neg a) = \neg(\neg a) = a$. \\
\noindent
(2) For $k=2$ follows by Definition \ref{th: algebre di Bochvar2}; $k=0$ is immediate. For $k=1$, $a\Juno a\lor\neg\Juno a = \neg(\Jdue a\lor\Jzero a)\lor\Jdue a\lor\Jzero a = (\neg\Jdue a\land\neg\Jzero a)\lor(\Jdue a\lor\Jzero a) =(\neg\Jdue a \lor\Jdue a\lor\Jzero a)\land(\neg\Jzero a \lor\Jdue a\lor\Jzero a) = (1\lor\Jzero a)\land(1\lor\Jdue a) = (\Jzero a\lor\neg\Jzero a\lor\Jzero a)\land(\Jdue a\lor\neg\Jdue a\lor\Jdue a) = (\Jzero a\lor\neg\Jzero a)\land(\Jdue a\lor\neg\Jdue a) = 1\land 1 = 1$.\\
\noindent
(3) Follows from (2) (and De Morgan laws). \\
\noindent
(4) $\Jdue(a\vee\neg a) = (\Jdue a\land\Jdue \neg a)\lor(\Jdue\neg a\land\Jdue\neg a)\lor(\Jdue a\land \Jdue a) = ((\Jdue a\land\Jdue \neg a)\lor \Jdue\neg a )\lor \Jdue a  = ((\Jdue a\lor \Jdue\neg a)\land (\Jdue \neg a\land \Jdue \neg a)) \lor \Jdue a = ((\Jdue a\lor \Jdue\neg a)\land \Jdue \neg a ) \lor \Jdue a = (\Jdue a\lor \Jdue\neg a \lor\Jdue a)\land (\Jdue \neg a  \lor \Jdue a) = (\Jdue a\lor \Jdue\neg a)\land (\Jdue \neg a  \lor \Jdue a) = \Jdue a\lor \Jdue\neg a$, where we have used \eqref{BCA:18} and distributivity.\\
\noindent
(5) $k = 2$: $\Jdue\Jdue a = \Jzero\neg\Jdue a = \neg\neg\Jdue a = \Jdue a$ (where we have used \eqref{BCA:10}).\\
\noindent
$k = 0$: $\Jdue\Jzero a = \Jdue\Jdue \neg a = \Jdue \neg a = \Jzero a$.\\
\noindent
$k = 1$: $\Jdue\Juno a = \Jdue\neg (\Jdue a\vee\Jzero a) = \Jzero (\Jdue a\vee\Jzero a) = \Jzero (\Jdue a\vee\Jdue\neg a) = \Jzero\Jdue(a\vee\neg a) = \neg\Jdue(a\vee\neg a) = \neg(\Jdue a\vee\Jdue\neg a) = \neg (\Jdue a\vee\Jzero a) = \Juno a$, where we have use the previous (4).\\
\noindent
(6) $k=2$ is included in Definition \ref{th: algebre di Bochvar2}, $k=0$ follows immediately by \eqref{BCA:10}. For $k= 1$: $\Jzero\Juno a = \Jdue\neg\Juno a = \Jdue (\Jdue a \vee\Jzero a) = (\Jdue \Jdue a\wedge \Jdue\Jzero a)\lor (\Jdue \neg\Jdue a\wedge \Jdue\Jzero a)\lor (\Jdue \Jdue a\wedge \Jdue\neg\Jzero a) = (\Jdue a\wedge \Jzero a)\lor (\Jzero\Jdue a\wedge \Jzero a)\lor (\Jdue a\wedge \Jzero\Jzero a) = (\Jzero a\land(\Jdue a\lor\neg\Jdue a))\vee (\Jdue a\wedge \neg\Jzero a) =  (\Jzero a\land 1)\vee (\Jdue a\wedge \neg\Jzero a) =\Jzero a\vee (\Jdue a\wedge \neg\Jzero a) = (\Jzero a\vee \Jdue)\land (\Jzero a\vee \neg\Jzero a) = (\Jzero a\vee \Jdue)\land 1 = \Jzero a\vee \Jdue = \neg\Juno a$.\\
\noindent
(7) We only have to show the case $\Jzero\varphi\approx\neg(\Jdue\varphi\lor\Juno\varphi)$ (as the others hold by Definition \ref{th: algebre di Bochvar2} and by the definition of $\Juno$). $\Jzero a =\Jdue\neg a = \neg(\Jzero\neg a\lor\neg\Juno\neg a) = \neg(\Jdue a\lor\Juno a).$\\
\noindent
(8) The case $k = 0$ follows immediately by the case $k=2$ (which holds by Definition \ref{th: algebre di Bochvar2}. For $k=1$: $\Juno a\vee\neg\Juno a = \neg(\Jdue a\vee\Jzero a)\vee (\Jdue a \vee\Jzero a) = (\neg\Jdue a\wedge\neg\Jzero a)\vee (\Jdue a \vee\Jzero a) =(\neg\Jdue a\vee \Jdue a \vee\Jzero a)\wedge (\neg\Jzero a\vee\Jdue a \vee\Jzero a) = (1\vee\Jzero a)\wedge(1\vee\Jdue a) = (\Jzero a\vee\neg\Jzero a\vee\Jzero a)\wedge(\Jdue a\vee\neg\Jdue a\vee\Jdue a) = (\Jzero a\vee\neg\Jzero a)\wedge(\Jdue a\vee\neg\Jdue a) = 1\wedge 1 = 1.$\\
\noindent
(9) We just show the case $i = 2$, $k = 0$ (as the others are analogous). $\Jdue a\wedge (\Jdue a\vee\Jzero a) = \Jdue a\wedge \neg\Juno a = \neg(\Jzero a\vee\Juno a)\wedge\neg\Juno a = \neg\Jzero a\wedge\neg\Juno a\wedge\neg\Juno a = \neg\Jzero a\wedge\neg\Juno a = \neg(\Jzero a\vee\Juno a) = \Jdue a$. \\
\end{proof}

Observe that, by Lemma \ref{lemma: aritmetica 1} (in particular, (8) and (9)), it follows that the image $\Jdue (\A)$ (and hence of $\Jzero$ and $\Juno$) of a Bochvar algebra forms the universe of a Boolean algebra : a fact that we will use several times (in the proofs) of the next Lemma, where we will indicate with $\leq$ the order of the mentioned Boolean algebra.

\begin{lemma}\label{lemma: aritmetica 2}
The following identities and quasi-identities hold in $\class{BCA}_2$.
\begin{enumerate}
\item $\Jdue 1 \approx 1$, $\Jzero 0 \approx 1$, $\Jdue 0 \approx 0$, $\Jzero 1 \approx 0$.


\item $\Jdue \varphi\vee\Jdue \varphi\approx \Jdue(1\vee \varphi)$.
\item $\Jdue(1\vee \varphi)\approx\Jdue(1\vee\neg \varphi)$.
\item $\Jdue(1\vee \varphi)\approx(\Jdue(1\vee (0\land\varphi))$.

\item $J_{_i}\varphi\leq \neg J_{_k}\varphi$, for every $i\neq k\in\{0,1,2\}$.
\item $\Jdue(\varphi\land 0)\approx 0$.

\item $\Jdue(1\vee(\varphi\land\psi))\approx\Jdue(1\vee \varphi)\land\Jdue(1\vee \psi)$.

\item $\Juno(\varphi\land\psi)\approx\Juno\varphi\vee\Juno\psi$. \label{eq: Juno meet}
\item $\Jzero(\varphi\land\psi)\approx (\Jdue\varphi\land\Jzero\psi)\vee(\Jzero\varphi\land\neg\Juno\psi).$ \label{eq: Jzero meet}

\item $\Jdue(\varphi\land\psi)\approx \neg(\Jdue\varphi\land\Jzero\psi)\land\Jdue\varphi\land\neg\Juno\psi$.
\item $\Jdue(\varphi\land\psi)\approx\Jdue\varphi\land\Jdue\psi$.


\item $J_{_0}(\varphi\lor \psi)\approx J_{_0}\varphi\land J_{_0}\psi$.
\item $ \varphi\lor J_{_k}\varphi \approx \varphi$, for $k\in\{1,2\}.$
\item $\Jzero \varphi \thickapprox \Jzero \psi \;\&\; \Juno \varphi \thickapprox \Juno \psi  \;\&\; \Jdue \varphi \thickapprox \Jdue \psi \;\Rightarrow\; \varphi \thickapprox \psi$.
\end{enumerate}

\end{lemma}
\begin{proof}
Let $\A\in\class{BCA}_2$ and $a,b\in A$.
\begin{enumerate}
\item $\Jdue 1 = \Jdue (\Jdue a\vee\neg\Jdue a) = (\Jdue\Jdue a\vee \Jdue\neg\Jdue a)\wedge (\Jdue\neg\Jdue a\vee \Jdue\neg\Jdue a) \vee (\Jdue\Jdue a\vee \Jdue\Jdue a) = (\Jdue a\wedge\Jzero\Jdue a) \vee \Jzero\Jdue a\vee\Jdue a = (\Jdue a\wedge\neg\Jdue a) \vee \neg\Jdue a\vee\Jdue a = 0 \vee \neg\Jdue a\vee\Jdue a = 0\vee 1 = 1$.\\
\noindent
$\Jdue 0 = \neg(\Jzero 0\vee\Juno 0) =\neg(1 \vee \Juno 1) = \neg 1 = 0$. The last equality follows from this one. \\





\item $\Jdue(1\vee a) = (\Jdue 1\wedge\Jdue a)\vee(\Jdue 1\wedge\Jzero a)\vee (\Jdue 0\wedge\Jdue a) = (1\wedge\Jdue a)\vee(1\vee\Jzero a)\vee (0\wedge\Jdue a) = \Jdue a\vee\Jzero a\vee 0 = \Jdue a\vee\Jzero a$.\\

\item It follows directly from the previous point, upon observing that $\Jdue\varphi\approx\Jzero\neg\varphi$. \\

\item Observe that $1\vee(0\wedge a) = (1\vee 0)\land (1\vee a) = 1\land(1\vee a) = 1\vee a$ (by Lemma \ref{lemma: aritmetica 1}-(1)).\\

\item Immediate from Lemma \ref{lemma: aritmetica 1}-(7).\\

\item Observe that, by the previous point, $\Jdue (0\land a)\leq\neg\Jzero(0 \land a) = \neg\Jdue(1\vee \neg a) = \neg\Jdue(1\vee a)$. Therefore $\Jdue (0\land a)\leq \neg\Jdue(1\vee a) =  \neg\Jdue(1\vee (0\wedge a)) = \neg(\Jdue (0\wedge a)\vee\Jzero (0\wedge a)) = \neg\Jdue (0\wedge a)\land\neg\Jzero (0\wedge a)\leq \neg\Jdue (0\wedge a) $, hence $\Jdue (0\land a) = 0$.\\

\item Applying De Morgan laws and \ref{BCA:18}, we have:
\begin{align*}
\Jdue(1\vee(a\land b)) =& \Jdue(1\vee\neg a\vee\neg b)
\\ = &\Jdue((1\vee\neg a)\vee\neg b)
\\ = & (\Jdue (1\vee\neg a) \wedge \Jzero b)\vee (\Jdue (1\vee\neg a) \wedge\Jdue a)\vee(\Jzero (1\vee\neg a) \wedge\Jzero b) 
\\ = & (\Jdue (1\vee a) \wedge \Jzero b)\vee (\Jdue (1\vee a) \wedge\Jdue a)\vee(\Jdue (0\land a) \wedge\Jzero b) 
\\ = & (\Jdue (1\vee a) \wedge \Jzero b)\vee (\Jdue (1\vee a) \wedge\Jdue a)\vee 0
\\ = & (\Jdue (1\vee a) \wedge \Jzero b)\vee (\Jdue (1\vee a) \wedge\Jdue a)
\\ = & \Jdue (1\vee a) \land(\Jzero b\vee \Jdue b)
\\ = & \Jdue (1\vee a) \land \Jdue (1\vee b).
\end{align*} \\

\item The claim is equivalent to (7). Indeed $\Juno(\varphi\land\psi)\approx\Juno\varphi\vee\Juno\psi$ iff $\neg\Juno(\varphi\land\psi)\approx\neg\Juno\varphi\land\neg\Juno\psi$ iff $\Jdue(\varphi\land\psi)\vee\Jzero (\varphi\land\psi)\approx (\Jdue\varphi\lor\Jzero\varphi)\land(\Jdue\psi\lor\Jzero\psi)$ iff $\Jdue(1\vee(\varphi\land\psi))\approx\Jdue(1\vee \varphi)\land\Jdue(1\vee \psi)$. \\

\item Easy calculation using \eqref{BCA:18}, De Morgan laws, distributivity and Lemma \ref{lemma: aritmetica 1}-(6).\\

\item Lemma \ref{lemma: aritmetica 1}-(6), we have 
\begin{align*}
\neg\Jdue (a\wedge b) =& \Jzero (a\wedge b)\vee\Juno (a\wedge b)
\\ = & (\Jdue a\wedge\Jzero b)\vee(\Jzero a\wedge\neg\Juno b)\vee\Juno (a\wedge b) & (\ref{eq: Jzero meet})
\\ = & (\Jdue a\wedge\Jzero b)\vee(\Jzero a\wedge\neg\Juno b)\vee\Juno b\vee\Juno a & (\ref{eq: Juno meet}) 
\\ = & (\Jdue a\wedge\Jzero b)\vee((\Jzero a \vee\Juno b)\wedge(\neg\Juno b\vee\Juno b))\vee\Juno a &
\\ = & (\Jdue a\wedge\Jzero b)\vee((\Jzero a \vee\Juno b)\wedge 1)\vee\Juno a &
\\ = & (\Jdue a\wedge\Jzero b)\vee \Jzero a \vee\Juno b\vee\Juno a &
\\ = & (\Jdue a\wedge\Jzero b)\vee \neg\Jdue a \vee\Juno b, &
\end{align*} 
thus the conclusion follows by De Morgan laws.\\

\item By the previous point, we have
\begin{align*}
\Jdue (a\wedge b) =& \neg(\Jdue a\wedge\Jzero b)\wedge\Jdue a\wedge\neg\Juno b
\\ = & (\neg\Jdue a \vee \neg\Jzero b)  \wedge \Jdue a \wedge \neg\Juno b 
\\ = &((\neg\Jdue a\wedge \Jdue a)\vee (\neg\Jzero b  \wedge \Jdue a)) \wedge \neg\Juno b
\\ = &( 0\vee (\neg\Jzero b  \wedge \Jdue a)) \wedge \neg\Juno b  
\\ = &\neg\Jzero b  \wedge \Jdue a \wedge \neg\Juno b
\\ = &\Jdue a \wedge \Jdue b.
\end{align*} 

\item $\Jzero(a\vee b) = \Jdue(\neg a\wedge\neg b) = \Jdue\neg a\wedge\Jdue\neg b = \Jzero a\wedge\Jzero b$. \\

\item We show that the antecedent of the quasi-identity \eqref{BCA:quasi} is satisfied, so is the consequent. $\Jdue (a\vee\Jdue a) = (\Jdue a\land\Jdue\Jdue a)\lor(\Jdue\neg a\wedge\Jdue\Jdue a)\lor(\Jdue a\land\Jdue\neg\Jdue a) = (\Jdue  a\land\Jdue a)\vee(\Jzero a\land\Jdue a)\vee (\Jdue a\land\Jdue\Jzero a) = \Jdue a\vee(\Jzero a\land\Jdue a)\vee (\Jzero a\land\Jdue a) = \Jdue a\vee(\Jzero a\land\Jdue a) = \Jdue a $, where in the last passage we have used the dual version of (9).\\
\noindent
$\Jzero(a\vee \Jdue a) = \Jzero a\land\Jzero\Jdue a = \Jzero a\land\neg\Jdue a = \Jzero a\land(\Jzero a\lor\Juno a) = \Jzero a$. Thus, by the quasi-identity \eqref{BCA:quasi} we have the conclusion. \\
\noindent
The case of $k=1$ is proved analogously.\\

\item We just have to show that $\Jzero \varphi \thickapprox \Jzero \psi  \;\&\; \Jdue \varphi \thickapprox \Jdue \psi $ implies $\Jzero \varphi \thickapprox \Jzero \psi \;\&\; \Juno \varphi \thickapprox \Juno \psi  \;\&\; \Jdue \varphi \thickapprox \Jdue \psi$. Suppose $\Jzero a = \Jzero b$ and $\Jdue a = \Jdue b$. Then $\Juno a = \neg(\Jdue a\vee\Jzero a) = \neg(\Jdue b\vee\Jzero b) = \Juno b $.
\end{enumerate}
\end{proof}

\begin{proof}[Proof of Theorem \ref{th: algebre di Bochvar2}]
The original  axiomatization of $\class{BCA}$ (Definition \ref{def: algebre di Bochvar}) includes all the identities (1)-(13); all the remaining identities (and quasi-identities) appearing in Definition \ref{def: algebre di Bochvar} but not in Theorem \ref{th: algebre di Bochvar2} have been shown to follow, from the axiomatization provided in Theorem \ref{th: algebre di Bochvar2} in Lemmas \ref{lemma: aritmetica 1} and \ref{lemma: aritmetica 2}.
\end{proof}

%

\noindent
\textbf{Ackowledgements.} S. Bonzio acknowledges the support by the Italian Ministry of Education, University and Research through the PRIN 2022 project DeKLA (``Developing Kleene Logics and their Applications'', project code: 2022SM4XC8), the support of the GOACT project - funded by Fondazione di Sardegna -- and partial support by the MOSAIC project (H2020-MSCA-RISE-2020 Project 101007627). He gratefully acknowledges also the support of the INDAM GNSAGA (Gruppo Nazionale per le Strutture Algebriche, Geometriche e loro Applicazioni). The work of Michele Pra Baldi was partially founded by the Juan de la Cierva fellowship 2020 (FJC2020-044271-I). He also acknowledges the  CARIPARO Foundation excellence project (2020-2024): ``Polarization of irrational collective beliefs in post-truth societies''.


\end{document}